\documentclass[12pt]{amsart}
\usepackage{geometry}
\usepackage[utf8]{inputenc}
\usepackage{amsmath}
\usepackage{amsthm}
\usepackage{listings}

\usepackage{graphicx}
\usepackage{subcaption} 

\usepackage{xcolor}
\usepackage{bbm}
\usepackage{amsfonts,amssymb,bm}
\usepackage{fancyhdr}
\usepackage{mathrsfs}
\usepackage[normalem]{ulem}
\useunder{\uline}{\ul}{}
\usepackage{mathtools}
\usepackage{appendix}
\usepackage{tikz,tikz-cd}
\usepackage{graphicx}
\usepackage[all]{xy}

\usepackage{float} 
\usepackage{amsthm}

\newcommand{\hvol}{{\widehat{\rm vol}}}
\newtheorem{thm}{Theorem}[section]
\newtheorem{defi}[thm]{Definition}
\newtheorem{lem}[thm]{Lemma}
\newtheorem{cor}[thm]{Corollary}
\newtheorem{prop}[thm]{Proposition}

\newtheorem{conj}[thm]{Conjecture}
\newtheorem{ques}[thm]{Question}
\newtheorem{fact}{Fact}

\theoremstyle{definition}
\newtheorem{rmk}[thm]{Remark}
\newtheorem{ex}[thm]{Example}

\usepackage{times}
\usepackage{hyperref}
\def\ord{{\rm ord}}
\def\min{{\rm min}}
\def\max{{\rm max}}
\def\inf{{\rm inf}}
\def\sup{{\rm sup}}
\def\lim{{\rm lim}}
\def\limsup{{\rm lim\,sup}}

\def\dif{{\rm d}}

\def\bl{{\rm Bl}}

\def\spec{{\rm Spec}}

\def\log{{\rm log}}
\def\vol{{\rm vol}}
\def\Val{{\rm Val}}

\def\ord{{\rm ord}}

\def\coeff{{\rm coeff}}

\def\mult{{\rm mult}}
\def\val{{\rm Val}}

\def\bc{{\rm Bc}}

\def\sym{{\rm Sym}}

\def\ext{{\rm Ext}}

\def\BP{\mathbf{P}}

\def\Bv{\mathbf{v}}

\def\tX{\tilde{X}}
\def\tY{\tilde{Y}}
\def\tC{\tilde{C}}
\def\hX{\hat{X}}

\def\bz{{\mathbf{z}}}

\newcommand{\IC}{{\mathbb C}}

\newcommand{\IN}{{\mathbb N}}
 
\newcommand{\IP}{{\mathbb P}} 
\newcommand{\IQ}{{\mathbb Q}} 
\newcommand{\IR}{{\mathbb R}}

\newcommand{\IZ}{{\mathbb Z}}

\newcommand{\CC}{{\mathcal C}}
 
\newcommand{\CE}{{\mathcal E}}

\newcommand{\CI}{{\mathcal I}}
 
\newcommand{\CK}{{\mathcal K}}
\newcommand{\CL}{{\mathcal L}}
\newcommand{\CM}{{\mathcal M}} 

\newcommand{\CO}{{\mathcal O}}

\newcommand{\CU}{{\mathcal U}}

\newcommand{\CX}{{\mathcal X}}

\newcommand{\lee}{\leqslant}
\newcommand{\gee}{\geqslant}

\lstdefinelanguage{Mathematica}{
  morekeywords={Sum, E},
  sensitive=true,
  morecomment=[l](*),
  morestring=[b]"
}

\title{On the volume of K-semistable Fano manifolds}

\author{Chi Li}
\address{Chi Li: Department of Mathematics, Rutgers University, Piscataway, NJ 08854-8019, US}
\email{chi.li@rutgers.edu}

\author{Minghao Miao}
\address{Minghao Miao: School of Mathematics, Nanjing University, Nanjing 210093, China}
\email{minghao.miao@smail.nju.edu.cn}
\date{\today}

\begin{document}

\begin{abstract}

We prove that the anti-canonical volume of an $n$-dimensional K-semistable Fano manifold $X$ that is not $\IP^n$ 
is at most $2n^n$. Moreover, the volume is equal to $2n^n$ if and only if $X\cong \IP^1\times \IP^{n-1}$ or $X$ is a smooth quadric hypersurface $Q\subset \IP^{n+1}$. Our proof is based on a new connection between K-semistability and minimal rational curves. 


\end{abstract}

\maketitle
\tableofcontents

\section{Introduction}

The study of K-(poly)stable Fano varieties is a very active research area. Thanks to the resolution of the Yau-Tian-Donaldson conjecture, we know that a Fano variety admits a K\"{a}hler-Einstein metric (resp. a unique K\"{a}hler-Einstein metric) if and only if it is K-polystable (resp. K-stable). A Fano variety is K-semistable if and only if it degenerates (via a special test configuration) to a K-polystable Fano variety. There are well-established valuative criterion for K-(semi)stability (\cite{Fuj19, Li17}) and powerful methods to test it (see \cite{AZ22}). 

In this paper, we study the anticanonical volume (or degree) of K-semistable Fano varieties. For a Fano variety $X$, its (anticanonical) volume $\vol(X)$ is defined to be the self-intersection number $(-K_X)^n$. It is known that the volume of an $n-$dimensional Fano manifold can exceed the volume of the complex projective space $\IP^n$, even among toric Fano manifolds (see \cite[Section 5.11]{Deb01}). 
In a major advance in the study of K-stability, Fujita \cite{Fuj18} proved that the volume of an $n$-dimensional K-semistable Fano manifold $X$ satisfies $\vol(
X)\le (n+1)^n=\vol(\IP^n)$ and the first equality holds if and only if $X\cong \IP^n$. It was later generalized to possibly singular $\IQ$-Fano varieties by Liu \cite{Liu18}. The toric Fano case was proved earlier in \cite{BB17}. Moreover, from the boundedness of K-semistable $\IQ$-Fano varieties (\cite[Corollary 1.2]{Jia17}), we know that the set of volumes of $n$-dimensional K-semistable $\IQ$-Fano varieties is finite away from zero. In this paper, we 
solve the conjecture on characterizing the second-largest volume for K-semistable Fano manifolds,
\begin{thm}\label{conj:2ndVol}\label{thm-main}
Any K-semistable Fano manifold $X$ that is not $\IP^n$ satisfies $(-K_X)^n\le 2n^n$ and the equality holds only if $X\cong \IP^1\times \IP^{n-1}$ or $X$ is a smooth quadric hypersurface $Q\subset \IP^{n+1}$. In particular, this holds for any K\"ahler manifold admitting K\"{a}hler-Einstein metric with positive Ricci curvature. 
\end{thm}
Like the result of Fujita (\cite{Fuj18}), the last statement could be seen as a result in K\"{a}hler geometry. However, our proof uses essentially algebro-geometric methods. The above result seems to be first conjectured in \cite[Problem 2.6]{AIM}, and is stated as a folklore conjecture in \cite[Conjecture 6.8]{Zhuang25} which highlights the interesting (but also mysterious) feature that there are two Fano manifolds with the second largest volume. Our proof will indeed give a satisfactory explanation of this feature by connecting it to the theory of minimal rational curves on Fano manifolds. 

There is a local analogue of this conjecture sometimes called the ODP conjecture (see Conjecture \ref{conj:ODP}) for the volume of klt singularity (see \cite{SS17,LX19}). 
The ODP conjecture would imply that the volume of singular K-semistable Fano varieties is strictly less than $2 n^n$ (by Fujita-Liu's local-to-global volume comparison, see Remark \ref{rem-sing}). Conversely, using Theorem \ref{thm-main}, we immediately verify the ODP conjecture for the Fano cone over a K-semistable smooth Fano manifold (see Theorem \ref{cor:ODPregular}). 


The toric version of Theorem \ref{thm-main} has its own interest and is called the ``gap hypothesis" in \cite[Conjecture 3.10]{AB24}. It implies the sharp bound of a canonical height of canonical model of toric Fano varieties over $\spec \,\IZ$ by \cite[Lemma 3.8]{AB24}. 
Moreover, in the toric case, the ODP conjecture is solved by Moraga-Süß's result (\cite{MS24}) using convex geometric methods. So combined with Theorem \ref{thm-main}, we indeed have the sharp upper bound for all possibly singular K-semistable toric Fano varieties.  
\begin{thm}[See Theorem \ref{thm:vtoric}]\label{thm-toric}
Let $X$ be an $n$-dimensional K-semistable toric Fano variety that is not $\IP^n$. Then $(-K_X)^n\le 2n^n$ and the equality holds only if $X\cong \IP^1\times \IP^{n-1}$. 
\end{thm}
We remark that Theorem \ref{thm-toric} was proved in \cite{AB24} for smooth toric Fano manifolds
of dimension $n\lee 6$ by using the classification of \cite{toricdata} and also for certain singular toric Fano varieties by convex geometric arguments (see \cite[Lemma 3.9]{AB24}).
Furthermore, using the well-known one-to-one correspondence between Gorenstein toric Fano varieties and reflexive lattice polytopes in $\IR^n$, Theorem \ref{thm-toric} immediately implies a convex geometric statement.
\begin{cor}[See Corollary \ref{cor:polytope}]\label{cor-lattice}
    Suppose $P\subseteq \IR^n$ is an $n$-dimensional reflexive lattice polytope with its barycenter at $0\in\IR^n$. Assume $P$ is not unimodularly equivalent to $(n+1)$ times a standard simplex $(n+1)\Delta_n$. Then the volume of $P$ with respect to the Lebesgue measure in $\IR^n$ satisfies $\vol_{\IR^n}(P)\lee 2n^n/n!$ and the equality holds if and only if $P$ is unimodularly equivalent to $[0,2]\times (n\Delta_{n-1})$.
\end{cor}

We will prove Theorem \ref{thm-main} by exploring a new connection between K-stability and minimal rational curves. We use the following invariant 
$$
l_X=\min\{(-K_X\cdot C)|\, C\subset X \, \text{is a minimal rational curve on}\, X\}\in \{2,3,\cdots, \dim X+1\}.
$$ 
For the case of $l_X=\dim X+1$ and $l_X=\dim X$, we can use the classification result of \cite{CMS02, Miy04,CD15,DH17} and Proposition \ref{prop-lxn} to verify Theorem \ref{thm-main}.
At the other extreme, if $l_X=2$ and when the minimal rational curve is embedded, then the minimal rational curve has a trivial normal bundle. In this special case Theorem \ref{thm-main} is a consequence of the following generalization of a result of K. Fujita (\cite{Fuj18}). 

\begin{thm}\label{thm-ntrivial}
    Let $X$ be an $n$-dimensional K-semistable Fano manifold, and $Z\subset X$ be a codimension $r$ smooth complex submanifold. Assume that the normal bundle of $Z$ inside $X$ is trivial and set $d = (-K_X)^{n-r}\cdot Z$. Then we have the inequality 
    \begin{equation}\label{eq-volprod}
    (-K_X)^n \lee (r+1)^r\cdot \binom{n}{r}\cdot d=(-K_{\IP^r\times Z})^n.
    \end{equation}
    Moreover, the equality holds if and only if $X$ is biholomorphic to $\IP^r\times Z$. In particular in the equality case $Z$ is also a K-semistable Fano manifold. 
\end{thm}

To estimate the volume in much more involved cases when $3\leq d=l_X\leq n-1$, we are going to use the valuative criterion to test K-semistability via  special valuations centered at minimal rational curves. 
While the proof of Theorem \ref{thm-ntrivial} uses ordinary blowup of submanifolds, the proof of Theorem \ref{thm-main} needs weighted blowups (see Remark \ref{rmk-12}). 
In particular, we will define and analyze a particular weighted blowup of type $({1}^{\oplus (d-2)}, 2^{\oplus (n-d+1)})$ that is adapted to the splitting of normal bundle over $\IP^1$ (see \eqref{eq-split}). 
The analysis will also allow us to overcome the difficulty that the minimal rational curve may have singularities.  



We sketch the organization of this paper. In the next section, we recall basic knowledge about some key concepts used in later sections: K-stability, Seshadri constant, and minimal rational curves. In section \ref{sec-ntrivial}, we prove Theorem \ref{thm-ntrivial} which implies Theorem \ref{thm-main} in the special case when the Fano manifold contains an embedded minimal rational curve with a trivial normal bundle. We also deduce Theorem \ref{thm-toric} and Corollary \ref{cor-lattice}. In section \ref{sec-main}, we give the full proof of Theorem \ref{thm-main}. We end the paper by giving some related examples and finding (K-semistable) Fano manifolds with minimal volumes. 

\subsection*{Acknowledgments}
C. Li is partially supported by NSF (Grant No. DMS-2305296). M. Miao would like to thank his advisor Gang Tian for his constant support, encouragement, and guidance. M. Miao would like to thank Jiyuan Han, Xiaowei Jiang, Yang Liu, Yalong Shi, and Linsheng Wang for helpful discussions. We thank Robert Berman, Kento Fujita, Jie Liu, Yuchen Liu,  Chenyang Xu, Kewei Zhang and Ziquan Zhuang for their interest and helpful comments, and Kento Fujita for pointing to us the Fact \ref{fact-vol}. We are especially grateful to Ziquan Zhuang for kindly pointing out some error in our argument that appeared in the first version of this preprint (see Remark \ref{rmk-Zhuang}) and very helpful suggestions. 
This work is carried out during M. Miao's visit to Rutgers University. He would like to thank the graduate school of Nanjing University for providing funding support and thank the Department of Mathematics at Rutgers University for its hospitality. We are grateful to an anonymous referee for careful reading and valuable suggestions that greatly improve our paper. The authors thanks ChatGPT 5.6 for helping with the proof of Proposition \ref{prop-lxn} (see Remark \ref{rmk-Miyaoka}).




\section{Preliminaries}
We will work over $\IC$. Unless otherwise specified, all varieties are assumed to be normal and projective. A $\IQ$-Fano variety $X$ is a variety with at worst klt singularities such that the anti-canonical divisor $-K_X$ is an ample $\IQ$-Cartier divisor. A singularity $x\in X$ consists of a variety $X$ and a closed point $x \in X$. A singularity $x \in X$ is called klt if $X$ is klt near a neighborhood of $x$. A $\IR$-valuation over singularity $x$ is a valuation $v\colon K(X)^*\rightarrow \IR$ centered at $x$ (namely, for all $f \in \CO_{X,x}$, we have $v(f)\gee 0$  and $v(f)>0$ if and only if $f \in \mathfrak{m}_x$) and $v|_{\IC^*}=0$. The set of all such valuations is denoted by $\val_{X,x}$. 
\subsection{K-stability}
K-stability, first introduced by Tian (\cite{Tia97}) and later reformulated algebraically by Donaldson (\cite{Don02}), is an algebro-geometric notion to characterize the existence of K\"ahler-Einstein metrics on Fano varieties. In this subsection, we recall some notions in K-stability theory that are relevant to our paper, and refer to \cite{Xu23} for a detailed exposition of K-stability theory. 

We say that a prime divisor $E$ is over $X$ if there exists a proper birational morphism $\mu\colon Y\rightarrow X$ such that $Y$ is normal and $E$ is a prime divisor on $Y$. We define the log discrepancy of the divisor $E$ over $X$ as 
$$
A(E)=A_{X}(E)\coloneqq 1+\coeff_{E}(K_{Y}-\mu^*K_X).
$$ 
The volume of an $\IR$-Cartier divisor $D$ is defined as $\vol_X(D)\coloneqq \mathop{\limsup}_{m\rightarrow +\infty}\frac{h^0(X,\CO_X(\lfloor mD\rfloor))}{m^n/n!}$. Note that the limsup is actually a limit and the volume function is continuous in the big cone $\text{Big}(X)$ (\cite[Section 2.2.C]{Lazar04}). Define the $S$-invariant 
$$S(E):=S(-K_X;E)\coloneqq \frac{1}{(-K_X)^n}\int_0^{\infty}\vol(-K_X-xE)\, \dif x, $$
where, for the simplicity of notation, we just write $\vol(-K_X-xE)$ for $\vol(\mu^*(-K_X)-xE)$. 
We have the following valuative criterion for K-semistability,
\begin{thm}{\cite{Fuj19,Li17}} \label{thm:beta}
    A Fano variety $X$ is K-semistable if and only if 
    $A(E)-S(E)\gee 0$ for every divisor $E$ over $X$. 
\end{thm}
Another equivalent way to characterize K-semistability is via the  $\delta$-invariant (also known as the stability threshold), we recall $\delta(X)\coloneqq \inf_{E/X}\frac{A(E)}{S(E)}$ where the infimum is taking over all divisors $E$ over $X$ (\cite{FO16,BJ20}). Then $X$ is K-semistable if and only if $\delta(X)\gee 1$ by the valuative criterion. 

One may define the log discrepancy $A_{X}(v)$ for any valuation $v\in \val_{X,x}$, see \cite[Section 5.1]{JM12}. For a klt singularity $x\in X$, we always have $A_{X}(v)>0$ for any nontrivial valuation $v\in\Val_{X,x}$. Denote $\val^*_{X,x}=\{v\in \val_{X,x}\mid A_{X}(v)<+\infty\}$. The volume of a valuation $v\in \val_{X,x}$ is defined as
\begin{eqnarray*}
    \vol_{X,x}(v) \coloneqq \mathop{\limsup}_{m\rightarrow +\infty} \frac{l(\CO_{X,x}/\mathfrak{a}_m(v))}{m^n/n!},
\end{eqnarray*} 
where $\mathfrak{a}_m(v)$ denotes the valuation ideals:  $\mathfrak{a}_m(v)\coloneqq \{f\in \CO_{X,x}\mid v(f)\geq m\}$. The first named author introduced the following invariant for singularity (\cite{Li18}), which plays a key role in the study of local K-stability:
\begin{defi}
    Suppose $x \in X$ is a klt singularity, for a valuation $v \in \val_{X,x}$, we define the normalized volume 
    \begin{eqnarray*}
        \hvol_{X}(v)\coloneqq 
        \left\{\begin{aligned}
            &A_{X}(v)^n\cdot \vol(v)&\text{if}\ & A_X(v)<+\infty \\
            & +\infty & \text{if}\ & A_X(v)=+\infty.
        \end{aligned}\right.
    \end{eqnarray*}
    The local volume of $x \in X$ is defined as $\hvol(x,X)\coloneqq \inf_{v\in \Val_{X,x}} \hvol_{X}(v)$.
\end{defi}
It was proved that for any $n$-dimensional klt singularity $x\in X$, we have $\hvol(x,X)\lee n^n$, with the equality holds if and only if $x\in X$ is a smooth point (\cite{LX19}). The following conjecture is known as the ODP conjecture (\cite{SS17,LX19}).
\begin{conj}\label{conj:ODP}(ODP conjecture)
    The second largest local volume of an $n$-dimensional klt singularity is $2(n-1)^n$. Moreover, the local volume is equal to $2(n-1)^n$ if and only if $x\in X$ is an ordinary double point.
\end{conj}
For later purposes, we formulate a calculus lemma and an estimate of volume. 
Let $\mathscr{F}$ denote the collection of continuous functions $\phi: [0, \infty)\rightarrow [0, \infty)$ that is piecewise smooth, strictly increasing and surjective. In particular, such a function $\phi$ satisfies 
 \begin{equation}\label{eq-0inflim}
       \mathop{\lim}_{x\rightarrow 0^+}\phi(x)=0, \quad \mathop{\lim}_{x\rightarrow+\infty}\phi(x)=+\infty. 
    \end{equation}

For each $x\in [0, \infty)$, we will use $\phi'(x)$ to denote the right derivative at $x$:
\begin{equation*} \phi'(x)=\lim_{\epsilon\rightarrow 0^+}
\frac{\phi(x+\epsilon)-\phi(x)}{\epsilon}. 
\end{equation*}
For any $\phi\in \mathscr{F}$, its inverse function $\phi^{-1}: [0, \infty)\rightarrow [0, \infty)$ is also an element of $\mathscr{F}$.  
For any $\phi\in \mathscr{F}$, consider the following function of $V\in (0, \infty)$:
 \begin{equation*}
       F^{\phi}(V):= 
       \frac{1}{V}\int_0^{\phi^{-1}(V)}\left(V-\phi(x)\right)\dif x. 
    \end{equation*} 
\begin{lem}  \label{lem-cal}
(1) For any $\phi\in \mathscr{F}$, the function $F^\phi(V)$ is piecewise smooth, strictly increasing and satisfies the limit conditions in \eqref{eq-0inflim}. As a consequence, we can consider it as an element of $\mathscr{F}$. 

(2) If $\psi\in \mathscr{F}$ satisfies $(\phi^{-1})'(V)\ge (\psi^{-1})'(V)$ for any $V\in (0, \infty)$, then $F^{\phi}(V)\ge F^{\psi}(V)$ for any $V\in (0, \infty)$. If moreover there exists $V_1>0$ and $\epsilon>0$ such that $(\phi^{-1})'(V)>(\psi^{-1})'(V)$ for $V\in (V_1, V_1+\epsilon)$, then $F^\phi(V)>F^{\psi}(V)$ for $V> V_1$. As a consequence, we get $(F^{\phi})^{-1}(A)\le (F^{\psi})^{-1}(A)$ for any $A\in (0, \infty)$ and with a strict inequality if $A>F^\psi(V_1)$. 
\end{lem}
\begin{proof}
    First assume that $\phi(x)$ is smooth and strictly increasing on $[0, \infty)$. Then its inverse function $\phi^{-1}(x)$ is also smooth, strictly increasing and satisfies the limit conditions in \eqref{eq-0inflim}. 
    Set $G(V)=\int_0^{\phi^{-1}(V)}(V-\phi(x))\dif x$ so that $F(V)=F^{\phi}(V)=G(V)/V$. Then $G$ is a smooth function of $V\in [0, \infty)$ satisfying $G(0)=0$. Its derivative is equal to $G'(V)=\int_0^{\phi^{-1}(V)}1 \, \dif x=\phi^{-1}(V)$. We can then calculate the derivative of $F(V)$ at any $V\in (0, \infty)$: 
    \begin{equation}\label{eq-dFdV}
        F'(V)=\frac{G'\cdot V-G}{V^2}=\frac{\int_0^{\phi^{-1}(V)}\phi(x)\dif x}{V^2}=\frac{H(V)}{V^2}>0
    \end{equation}
    with $H(V)=H^\phi(V)=\int_0^{\phi^{-1}(V)}\phi(x)\dif x$. So we know that $F(V)$ is a strictly increasing function of $V\in [0, \infty)$. Using L'Hospital's rule, we easily see that $F$ satisfies the limit conditions in \eqref{eq-0inflim}. So if we define $F(0)=0$, then $F\in \mathscr{F}$.  
   
   Note that $H=H(V)$ is a continuous function on $[0, \infty)$ satisfying $H(0)=0$ and $H'(V)=\phi(\phi^{-1}(V))(\phi^{-1})'(V)=V (\phi^{-1})'(V)$. By using \eqref{eq-dFdV} and integrating the assumed inequality twice, we get
   $F^\phi(V)\ge F^\psi(V)$. If moreover $(\phi^{-1})'(V)>(\psi^{-1})'(V)$ on an interval $(V_1, V_1+\epsilon)$, then we get the strict inequality $H^\phi(V)>H^\psi(V)$ for $V>V_1$ by the following estimate:
   \begin{eqnarray*}
       H^\phi(V)&=&H^\phi(V_1)+\int_{V_1}^V t (\phi^{-1})'(t)dt \\
       &>& H^\psi(V_1)+\int_{V_1}^V t (\psi^{-1})'(t)dt=H^\psi(V). 
   \end{eqnarray*}
  By using \eqref{eq-dFdV} and integrating once more, we get the strict inequality $F^\phi(V)> F^\psi(V)$ when $V>V_1$. 
   
   For the last statement, note that under the assumption of strict inequality, we know that if $A>F^\psi(V_1)$ then $(F^\psi)^{-1}(A)>V_1$ which implies:
   $F^\phi((F^\psi)^{-1}(A))>F^\psi((F^\psi)^{-1}(A))=A$. So we conclude that $(F^\psi)^{-1}(A)>(F^\phi)^{-1}(A)$ when $A>F^\psi(V_1)$. 

   If $\phi$ is only piecewise smooth, we can just carry out the above argument piecewise on each interval of smoothness. The same argument applies to prove the second statement too. 
\end{proof}

\begin{cor}\label{cor-calbd}
    Assume $(X, L:=-K_X)$ is a K-semistable $\IQ$-Fano variety. Let $E$ be a prime divisor over $X$ that satisfies the estimate: there exists $\phi\in \mathscr{F}$ such that for any $x\in [0, \infty)$, 
    \begin{equation}\label{eq-vollb}
        \vol(L-xE)\gee V-\phi(x). 
    \end{equation}
  Then there exists a unique solution $T\in (A_X(E), \infty)$ to the equation 
    \begin{equation}\label{eq-pseudoT}
        (T-A_X(E))\phi(T)=\Phi(T)
    \end{equation} 
    where $\Phi(x)=\int_0^x \phi(t)\dif t$ is the primitive function of $\phi(x)$ with $\Phi(0)=0$. 
    Moreover, we have an estimate of the volume: 
    \begin{equation}\label{eq-volbdT}
        V=(-K_X)^n\lee \phi(T)=(F^{\phi})^{-1}(A(E)).
    \end{equation} 
    The equality $V=\phi(T)$ holds if and only if the equality holds in \eqref{eq-vollb} for $x\in [0, T]$. In this case, $T$ is the pseudo-effective threshold of the valuation $\ord_E$ and in particular $\vol(L-TE)=0$. 
\end{cor}
\begin{proof}
    We have an estimate:
    \begin{eqnarray*}
        0&\lee& A_X(E)-\frac{1}{V}\int_0^{+\infty}\vol(L-xE)\dif x\lee A_X(E)-\frac{1}{V}\int_0^{\phi^{-1}(V)}(V-\phi(x))\dif x\\
        &=&A_X(E)-F^\phi(V).
    \end{eqnarray*}
    By the above lemma, $F^\phi(V)$ is a strictly increasing function that diverges to $+\infty$ as $V\rightarrow+\infty$. So there exists a unique $V^*$ that satisfies $A(E)=F^\phi(V^*)$ which is equivalent to the equality:
    \begin{equation*}
       (\phi^{-1}(V^*)-A(E))V^*=\int_0^{\phi^{-1}(V^*)}\phi(x)\dif x. 
    \end{equation*}
    Setting $T=\phi^{-1}(V^*)$, we see that $T$ satisfies 
    $(T-A(E))\phi(T)=\Phi(T)$. 
 \end{proof}

\begin{ex}[Fujita \cite{Fuj18}]
   Assume that $\phi(x)=x^n$ and $A(E)=n$. Then 
   $F^\phi(V)=\frac{n}{n+1}V^{1/n}$ and $\Phi(x)=\int_0^x\phi(t)dt=\frac{x^{n+1}}{n+1}$. The equation \eqref{eq-pseudoT} has the solution $T=n+1$, yielding $\phi(T)=(n+1)^n=(F^\phi)^{-1}(n)$. 
\end{ex}

\begin{ex}\label{ex-2nn}
   Assume that $\phi(x)=2n x^{n-1}$ and $A(E)=n-1$. Then  
   $F^\phi(V)=\frac{n-1}{n}\left(\frac{V}{2n}\right)^{1/(n-1)}$ and  $\Phi(x)=\int_0^x\phi(t)dt=2x^n$. The equation \eqref{eq-pseudoT} has the solution $T=n$, yielding $\phi(T)=2n^n=(F^\phi)^{-1}(n-1)$. 
\end{ex}

\subsection{Seshadri Constant} Let $X$ be a normal projective variety and $L$ an ample $\IQ$-Cartier divisor on $X$. 
Denote by $X^{\text{sm}}$ the set of smooth points of $X$. 
Let $Z\subset X^{\text{sm}}$ be a nonsingular closed subvariety of $X$ of codimension $r$. The Seshadri constant of $L$ at $Z$ is defined as $\epsilon(L,Z)\coloneqq \sup\{t\in \IR_{>0}\mid \pi^*L-tE\, \text{\,is ample}\}$, where $\pi\colon \bl_Z X\rightarrow X$ is the blowup of $X$ along $Z$ and $E\cong \IP(N_{Z/X}^{\vee})$ is the exceptional divisor. When $X = \IP^r\times Z$ and $Z$ is identified with $\{p\}\times Z$ for a fixed $p\in \IP^r$, $(\mathrm{Bl}_Z X, E)$ are the same as $((\mathrm{Bl}_{p}\IP^r)\times Z, E_p\times Z)$ where $E_p\cong \IP^{r-1}$ is the exceptional divisor of the blow up $\mathrm{Bl}_p \IP^r\rightarrow \IP^r$. We see that $\epsilon(-K_{\IP^r\times Z},Z) = r+1$. 

The next proposition generalizes a result of \cite{LZ18} to higher dimensional subvariety $Z$ with trivial normal bundle, which says there is a gap between $r$ and $r+1$ for all the possible value of $\epsilon(-K_X,Z)$.

\begin{prop}\label{prop-LZ}
    Suppose $X$ is a $\IQ$-Fano variety of dimension $n$, and let $Z \subset X^{\text{sm}}$ be a nonsingular subvariety of codimension $r\gee 1$. 
    Assume that the normal bundle of $Z$ is trivial, i.e., $N_{Z/X} \cong \CO_Z^{\oplus r}$, and the Seshadri constant $\epsilon(-K_X,Z)\in \IQ_{>0}$ satisfies $\epsilon(-K_X, Z)>r$.
    Then $X \cong \IP^r\times Z$ and $Z$ is identified with $\{p\}\times Z$ for a point $p\in \IP^r$.
\end{prop}

\begin{proof}
    We follow the argument as in \cite[Theorem 2]{LZ18}. Because the normal bundle is assumed to be trivial, by the adjunction formula $-K_Z=(-K_X)|_Z$, and hence $Z$ is a Fano manifold. 
    For simplicity, we denote $\epsilon\coloneqq \epsilon(-K_X,Z)>r$. 
    Let $\pi\colon \hat{X}\coloneqq\bl_Z X\rightarrow X$ be the blowup of $X$ along $Z$ and $E$ is the exceptional divisor. We denote $B\coloneqq \pi^*(-K_X)-\epsilon E$. From the definition of $\epsilon(-K_X,Z)$, we know $B$ is a nef divisor. We have $K_{\hat{X}} = \pi^*K_X+(r-1)E$, so $B-K_{\hat{X}} = 2(\pi^*(-K_X)-\frac{\epsilon+r-1}{2}E)$ is nef and big since $(\epsilon+r-1)/2<\epsilon$. Then by Kawamata's basepoint-free theorem (\cite[Theorem 3.3]{KM98}), we know $B$ is indeed semiample. Then there exists a fibration $\Phi\colon \hat{X}\rightarrow Y\subseteq \IP(H^0(\hat{X},kB))$ induced by the complete linear series $|kB|$ for some $k\gg 0$ and $Y$ is a closed subvariety.  Next, let $m$ be an integer such that $mB$ is Cartier. Note that 
    \begin{eqnarray*}
        mB-E-K_{\hat{X}} = (m+1)\left(\pi^*(-K_X)-\frac{m\epsilon+r}{m+1}E\right)
    \end{eqnarray*}
    is an ample divisor since $(m\epsilon+r)/(m+1)<\epsilon$. Then, by Kodaira's vanishing theorem, we get $H^1(\hat{X}, mB-E)=0$. Therefore, the natural map $H^0(\hat{X}, mB)\rightarrow H^0(E,mB|_E)$ is surjective for all $m>0$ such that $mB$ is Cartier. We conclude that $\Phi|_E\colon E\rightarrow Y$ is a closed embedding. Since $\hat{X}$ has klt singularites, then it has rational singularities (\cite[Theorem 5.22]{KM98}), so in particular it is Cohen-Macaulay. Since $B=-K_{\hat{X}}+(r-1-\epsilon)E\sim_{\Phi,\IQ}0$, we get $-K_{\hat{X}}\sim_{\Phi,\IQ} \lambda E$ where $\lambda=\epsilon-r+1>1$. By \cite[Lemma 8]{LZ18}, we have $\Phi\colon \hat{X}\rightarrow Y$ is not birational. Thus, $\Phi$ must be a fiber type contraction.

    Since we have already shown $\Phi|_E\colon E\rightarrow Y$ is a closed embedding, we conclude that $\Phi|_E\colon E\rightarrow Y$ is in fact an isomorphism, so $ Y\cong E \cong \IP(N_{Z/X}^{\vee})\cong Z\times \IP^{r-1}$. A general fiber of $\Phi$ is a smooth rational curve. By a similar argument as in \cite[Lemma 6]{LZ18}, we can show $\Phi\colon \hat{X}\rightarrow Y$ is a smooth $\IP^1$-fibration. Note that $s=\Phi|_E^{-1}\colon Y\rightarrow E$ gives a section of $\Phi$, then there exists a rank 2 vector bundle $\CE$ over $Y$ such that $\hat{X} = \IP_Y(\CE)$. By \cite[Proposition V.2.6]{Har77}, there exists an invertible sheaf $\CL$ on Y and a surjective morphism $\CE \rightarrow \CL$. We denote $\CK = \ker(\CE \rightarrow \CL)$. Then $\CO_{Y}(-1)\cong s^*N_{E/\hat{X}}\cong \CL \otimes \CK^{-1}$. We know that the choice of $\CE$ is unique up to twisting of an invertible sheaf. So we may assume $\CK = \CO_Y$, then $\CL = \CO_{Y}(-1)\cong p_2^* \CO_{\IP^{r-1}}(-1)$ with the projection $p_2\colon Y\rightarrow\IP^{r-1}$ and we have the following short exact sequence
    $$
0\rightarrow \CO_{Y}\rightarrow \CE\rightarrow \CO_Y(-1)\rightarrow 0.
$$
By \cite[Prop III.6.3]{Har77}, we have 
\begin{eqnarray*}
    \ext^1(\CL,\CO_Y) = H^1(Y,\CL^{\vee})
    =\bigoplus_{i+j=1}H^i(\IP^{r-1},\CO_{\IP^{r-1}}(1))\otimes H^j(Z,\CO_Z)=0,
\end{eqnarray*}
so the short exact sequence actually splits. Therefore, $\CE \cong \CO_Y \oplus \CO_Y(-1)$ and $\hat{X}\cong \IP(\CO_Y\oplus \CO_{Y}(-1))\cong \IP(\CO_{\IP^{r-1}}\oplus\CO_{\IP^{r-1}}(-1))\times Z\cong \bl_{p}\IP^r\times Z$ is isomorphic to $\IP^r\times Z$ blowup along $\{p\}\times Z$. So $X\cong  \IP^r \times Z$.
\end{proof}

\subsection{Minimal rational curves}

Since the celebrated work of Mori that introduces the bend-and-break method for constructing rational curves on Fano manifolds (\cite{Mor79}), the theory of rational curves has been developed extensively and has found numerous applications in algebraic geometry, especially in the classification of Fano manifolds with special properties. Here we recall some basic knowledge of rational curves and refer to \cite{Kol96} for detailed expositions. 

A free rational curve is represented by a morphism $f: \IP^1\rightarrow X$ that satisfies $f^*TX=\oplus_{i=1}^n \CO_{\IP^1}(a_i)$ with $a_i\ge 0$. A free rational curve is called minimal (or standard) if
$$
f^*TX=\CO_{\IP^1}(2)\oplus \CO_{\IP^1}(1)^{\oplus(d-2)}\oplus \CO_{\IP^1}^{\oplus (n-d+1)}. 
$$
Such a morphism $f$ must be an immersion. 
We will also be interested in its normal bundle:
$$
N_{f/X}:=f^*TX/T\IP^1=\CO(1)^{\oplus (d-2)}\oplus \CO^{\oplus (n-d+1)}.
$$
Any Fano manifold always admits minimal rational curves, which can be obtained from the bend-and-break process starting with a free rational curve (see \cite[IV.Theorem 2.10]{Kol96}). Denote by $\mathrm{RatCurves}^n(X)$ the normalization of open subset of $\mathrm{Ch}(X)$ parametrizing integral rational curves. 
An irreducible component $\CM$ of $\mathrm{RatCurves}^n(X)$ is referred to as a family of rational curves on $X$. 
The anticanonical degree $\deg(\CM)$ of the family $\CM$ is defined to be $-K_X\cdot C$ for any curve $C$ belonging to the family. This family $\CM$ is equipped with a $\IP^1$-bundle $p: \CU\rightarrow \CM$ and an evaluation morphism $q: \CU\rightarrow X$.  
The family $\CM$ is a \textit{dominating} family if the evaluation morphism $q: \CU\rightarrow X$ is dominant (i.e. has a dense image). This is equivalent to the condition that a general rational curve in $\CM$ is free. 
A dominating family $\CM$ is \textit{locally unsplit} if, for a general point $x\in X$, the subfamily $\CM_x=p(q^{-1}(x))$ parametrizing curves through $x$ is proper (i.e. compact). Note that if $\CM$ is a dominating family of rational curves on $X$ that has a minimal anticanonical degree, then $\CM$ is a dominating locally unsplit family. But not all dominating locally unsplit family has a minimal anticanonical degree.

Let $\CM$ be a dominating locally unsplit family of rational curves. By Mori's bend-and-break argument, we know that for a general point $x\in X$, a general curve in $\CM_x$ is minimal. In particular, $\deg(\CM)\in \{2, 3, \cdots, n+1\}$. Indeed, if a general free rational curve $C=[f]$ in $\CM_x$ is not minimal, then there are at least two $\CO(2)$ summands in the splitting of $f^*TX$. We can then fix two points on the curve $C$ and bend-and-break the rational curve into a non-integral (reducible) curve which would contradict the locally unsplit property. 
Following \cite{Miy04, CD15}, we set 
\begin{equation}\label{eq-lX1}
    l_X:=\min\{\deg \CM; \CM \text{ is a dominating locally unsplit family of rational curves on } X\}.
\end{equation}
By the above discussion, we see that (see \cite[Remark 4.2]{CD15}): 
\begin{eqnarray}\label{eq-lX}
    l_X&=&\min\{-K_X\cdot C; C\subset X \text{ is a free rational curve on } X\}\nonumber \\
    &=& \min\{-K_X\cdot C; C\subset X \text{ is a minimal rational curve on } X\}.
\end{eqnarray} 

Using the above notation, we have the following important classification results: 
\begin{thm}\label{thm-lX}
\begin{enumerate}
    \item (\cite{CMS02}) $l_X=n+1$ if and only if $X\cong \IP^n$. 
    \item (\cite{CD15}) If $n\geq 3$ and the Picard number $\rho(X)\ge 2$, then $l_X=n$ if and only if $X$ is the blowup of $\IP^n$ along a smooth codimension two subvariety $Y$ of degree $d_Y\in \{1,\dots, n\}$ that is contained in a hyperplane. 
\end{enumerate}
\end{thm}
\begin{rmk}\label{rmk-Miyaoka}
In previous versions of this paper, we used one of Miyaoka's results from \cite{Miy04} which says that 
when $\rho(X)=1$ and $l_X=n$, $X$ must be isomorphic to $Q^n$. Recently we are however notified that there may be a gap in Miyaoka's argument as pointed out by Dedieu-H\"{o}ring \cite{DH17} who proved a weaker result by assuming the equality $-K_X\cdot C\ge n$ for all rational curves $C$. It seems that Miyaoka's classification result for $\rho(X)=1$ and $l_X=n$ is not universally accepted (see reviews of the papers \cite{Miy04, CD15} on mathscinet of AMS, and also \cite[Remark 2.5]{Liu19v3}, \cite[Remark 8]{Wu23v3} for more information). Fortunately, a weaker result in the next proposition suffices for our purpose. Its proof is obtained with the assistance of ChaptGPT 5.6, and is contained in Appendix \ref{sec-lXn}. 
\end{rmk}


\begin{prop}\label{prop-lxn}
    Let $X$ be a smooth Fano manifold of dimension $n\ge 3$ with $l_X=n$ and Picard number $\rho(X)=1$. Then $(-K_X)^n\le 2n^n$. Moreover the equality holds if and only if $X\cong Q^n$. 
\end{prop}


\begin{ex}[\cite{HM03}]\label{ex-l2}
On the other extreme, when the Picard number is one, the condition $l_X=2$ is equivalent to the following condition:
     \begin{enumerate}
         \item There exists a minimal rational curve with trivial normal bundle. 
         \item For a general point $x\in X$, there are only finitely many rational curves through $x$ which have minimal degree with respect to $K_X^{-1}$. 
         \item for a general point $x\in X$, there exists a rational curve that has degree 2 with respect to $K_X^{-1}$.       
     \end{enumerate}
    There are many examples of such Fano manifolds. For example, except for the projective space $\IP^3$ and 3-dimensional hyperquadric $Q^3$, all Fano 3-folds of Picard number 1 satisfy this condition.
    Hypersurfaces of $\IP^{n+1}$ of degree $n$ or $n+1$ are also examples of such Fano manifolds when $n\gee 3$. 
\end{ex}

\section{Submanifolds with trivial normal bundles}\label{sec-ntrivial}


In this section, we will prove Theorem \ref{thm-ntrivial} by generalizing the methods from \cite{Fuj18, LZ18}, and apply it in the toric case to prove Theorem \ref{thm-toric}.
We start with a proposition that proves a more general version of the estimate \eqref{eq-volprod} by incorporating the $\delta$-invariant. 
\begin{prop}\label{prop:Ntrivial}
    Let $X$ be an $n$-dimensional Fano manifold, and $Z\subset X$ be a codimension $r$ non-singular subvariety with trivial normal bundle $N_{Z/X}=\CO_{Z}^{\oplus r}$. Set $d = (-K_X)^{n-r}\cdot Z$. Then we have the estimate:
    $$
    (-K_X)^n \lee \delta(X)^{-r}\cdot(r+1)^r\binom{n}{r}\cdot d=\delta(X)^{-r}\cdot( -K_{\IP^r\times Z})^n.
    $$  
\end{prop}
\begin{proof}
    Let $\pi\colon \hat{X}\coloneqq\bl_Z X \rightarrow X$ be the blowup of $X$ along $Z$ with the exceptional divisor $E$. First, it is clear that the log discrepancy 
    \begin{eqnarray*}
       A_X(E) = 1+ \coeff_{E}(K_{\hat{X}}-\pi^*K_X) = 1+(r-1) = r. 
    \end{eqnarray*}
    For the simplicity of notation, set $L=-K_X$. We can assume $x\in \IQ$ since the volume function $\vol_{\hat{X}}(\pi^*L-xE)$ is continuous. Then we take $k\in \IN^*$ sufficiently large such that $kx\in \IZ_{>0}$. Note that we have the exact sequence:
$$
0\rightarrow H^0(X, kL\otimes \CI_{Z}^{xk})\rightarrow H^0(X, kL)\rightarrow H^0(X, kL\otimes \CO_{xkZ})\rightarrow \cdots,
$$
which implies:
$$
h^0(X,kL\otimes \CI_Z^{xk})\gee h^0(X,kL)-h^0(X,kL\otimes \CO_{xkZ}).
$$
Note that the higher direct images $R^i\pi_*\CO_{\hat{X}}=0$ for $i>0$. By the Leray spectral sequences, we get $H^0(\hat{X},\pi^*(kL)-xkE)=H^0(X,kL\otimes \CI_Z^{xk})$ and $H^0(\hat{X},\pi^*(kL))=H^0(X,kL)$. Thus,
$$\vol_{\hat{X}}(\pi^*L-xE) = \mathop{\limsup}_{k\rightarrow +\infty} \frac{h^0(X,kL\otimes \CI_{Z}^{xk})}{k^n/n!}\gee L^n-\mathop{\limsup}_{k\rightarrow+\infty} \frac{h^0(X, kL\otimes \CO_{xkZ})}{k^n/n!}.$$
For  $j\in \IZ_{\gee 0}$, we use the exact sequence
$$
0\rightarrow \CI_Z^{j}/\CI_Z^{j+1}\rightarrow \CO_X/\CI_Z^{j+1}\rightarrow \CO_X/\CI_Z^{j} \rightarrow 0.
$$
And $H^1(X,kL\otimes (\CI_Z^{j}/\CI_Z^{j+1}))=\oplus H^1(Z, -kK_Z)=0$ from Kodaira vanishing and the assumption that normal bundle $N_{Z/X}$ is trivial. It is easy to prove by induction that:
\begin{eqnarray} \label{eq-tels}
h^0(X, kL\otimes \CO_X/\CI_Z^{xk})&=& h^0(X, kL\otimes \CO_X/\CI_Z^{xk-1})+h^0(X, kL\otimes \CI_Z^{xk-1}/\CI^{xk}_Z) \nonumber \\
&=&\sum_{j=0}^{xk-1}h^0(X, kL\otimes \CI_Z^j/\CI^{j+1}_Z). 
\end{eqnarray}
Using the assumption that $N_{Z/X}$ is trivial and  $Z$ is non-singular, by \cite[Thm \rm II. 8.24]{Har77}, we get:
\begin{eqnarray*}
    \CI_Z^{j}/\CI_Z^{j+1}&\cong&\mathrm{Sym}^j(\CI_Z/\CI_Z^2)=\mathrm{Sym}^j(N_{Z/X}^{\vee})=\CO_Z^{\oplus \binom{r-1+j}{r-1}}. 
\end{eqnarray*}
Then the right-hand-side of \eqref{eq-tels} is given by
\begin{eqnarray*}
    \sum_{j=0}^{xk-1}h^0(Z, \CO_X(kL)\otimes \CI_Z^j/\CI^{j+1}_Z) = \sum_{j=0}^{xk-1} \binom{j+r-1}{r-1}\cdot h^0(Z, -kK_Z) = \binom{r+xk-1}{r}\cdot h^0(Z, -kK_Z).
\end{eqnarray*}
Then,
\begin{eqnarray}\label{eq-h0}
    \mathop{\limsup}_{k\rightarrow +\infty} \frac{h^0(X,kL\otimes \CO_{xkZ})}{k^n/n!}&=&\mathop{\limsup}_{k\rightarrow +\infty}\frac{n!}{r!(n-r)!}\cdot\frac{(r+xk-1)!}{(xk-1)!\cdot k^r}\cdot\frac{h^0(Z,-kK_Z)}{k^{n-r}/(n-r)!} \nonumber \\
    &=&\binom{n}{r}(-K_Z)^{n-r}\cdot x^r.
\end{eqnarray}
We set $d = (-K_Z)^{n-r}=(-K_X)^{n-r}\cdot Z$. So we get $\vol_{\hat{X}}(\pi^*L-xE)\gee L^n - d \binom{n}{r}x^r$. 
Since $E\cong Z\times \IP^{r-1}$ and $\CO(-E)|_E=p_2^* \CO_{\IP^{r-1}}(1)$, one can easily see that the right-hand side equals to the top-intersection number 
\begin{eqnarray*}
    (\pi^*L-xE)^n &=& \sum_{k=0}^{n} \binom{n}{k} (\pi^*L)^{n-k}\cdot x^k(-E)^k\\
    &=&(\pi^*L)^n +(-1)^r \binom{n}{r} d x^r\cdot (-1)^{r-1}=L^n-d\binom{n}{r}x^r.
\end{eqnarray*}
Moreover, we have:
\begin{eqnarray*}
    S(-K_X;E) &=& \frac{1}{(-K_X)^n}\int_0^{T_X(E)} \vol_{\hat{X}}(\pi^*(-K_X)-xE) \,\dif x \\
    &\gee& \frac{1}{(-K_X)^n}\int_0^{\epsilon} \left((-K_X)^n - d\binom{n}{r}x^r\right) \,\dif x \\
    &=& \epsilon-\frac{1}{(-K_X)^n} \binom{n}{r} \frac{d\cdot \epsilon^{r+1}}{r+1} = \frac{r}{r+1}\epsilon,
\end{eqnarray*}
where $\epsilon = \left((-K_X)^n/(\binom{n}{r}\cdot d)\right)^{1/r}$. Then,
    \begin{eqnarray*}
        \delta(X) &\lee& \frac{A_X(E)}{S(-K_X;E)} \lee \frac{r}{\frac{r}{r+1}\epsilon} = \frac{r+1}{\left((-K_X)^n/(\binom{n}{r}\cdot d)\right)^{1/r}}.
    \end{eqnarray*}
     Therefore, $(-K_X)^n \lee \delta(X)^{-r}(r+1)^r\cdot \binom{n}{r}\cdot d = \delta(X)^{-r}\cdot (-K_{\IP^r\times Z})^n$.
\end{proof}
\begin{lem}\label{lem-seshadri}
    Under the same notation as the above proposition, we set 
    \begin{equation}\label{eq-LamZ}
    \Lambda_Z(L)\coloneqq \{x\in\IR_{\gee 0}\mid \vol_{\hat{X}}(\pi^*L-xE) = (\pi^*L-xE)^n\}.
    \end{equation}
    Then we have $\epsilon(L,Z) = \max\{t\in \IR_{\gee 0}\mid x\in \Lambda_Z(L) \,\, \text{for all}\, x\in[0,t]\}$.
\end{lem}
\begin{proof}
   The argument is similar to the proof of \cite[Theorem 2.3(2)]{Fuj18}. We denote 
   \begin{equation}\label{eq-gamZ}
   \gamma=\gamma_Z(L)\coloneqq \max\{t\in \IR_{\gee 0}\mid x\in \Lambda_Z(L) \,\, \text{for all}\, x\in[0,t]\}.
   \end{equation}
   When $0\lee t \lee \epsilon(L,Z)$, we have $\pi^*L-tE$ is nef, then $\vol_{\hat{X}}(\pi^*L-tE)=(\pi^*L-tE)^n$. Therefore, $\epsilon(L,Z)\lee \gamma$. In particular, $\gamma>0$. Now, in order to show $\epsilon(L,Z)\gee \gamma$, it suffices to show for any $\eta>0$ sufficiently small such that $\gamma-\eta\in \IQ_{>0}$ and $\pi^*L-(\gamma-\eta)E$ is ample. Fix an $\delta\in \IQ_{>0}$ such that $\pi^*L-\delta E$ is ample. Take any $t\in \IQ_{>0}$ satisfying $t\lee \min\{1,(\gamma-\eta)/\delta,\,\eta/(\gamma-\delta)\}$. Then,
    \begin{eqnarray*}
        \left(\pi^*L-(\gamma -\eta)E\right) - t(\pi^*L-\delta E) = (1-t) \left(\pi^*L-\frac{\gamma-\eta-t\delta}{1-t}E\right).
    \end{eqnarray*}
    We set $x_t\coloneqq (\gamma-\eta -t\delta)/(1-t)$. We take sufficiently large $k\in \IN^*$ such that $kx_t\in \IZ_{>0}$. From Kodaira's vanishing theorem, we have $H^i(X,kL)=0$ for $i\gee 1$. Then the exact sequence is given by
\begin{eqnarray*}
    0\rightarrow H^0(X,kL\otimes\CI_Z^{kx_t})\rightarrow  H^0(X,kL)\rightarrow H^0(X,kL\otimes \CO_{kx_tZ}) \rightarrow H^1(X,kL\otimes\CI_Z^{kx_t})\rightarrow 0.
\end{eqnarray*}
And $H^i(X,kL\otimes \CI_{Z}^{kx_t})=H^{i-1}(X,kL\otimes (\CO_X/\CI_Z^{kx_t}))$ for $i\gee 2$. On the one hand, since $x_t\in \Lambda_Z(L)$ and by the equation (\ref{eq-h0}),
\begin{eqnarray*}
    \mathop{\limsup}_{k\rightarrow+\infty}\frac{h^1(X,kL\otimes\CI_{Z}^{kx_t})}{k^n/n!} = -L^n+\vol_{\hat{X}}(\pi^*L-x_tE)+\mathop{\limsup}_{k\rightarrow+\infty}\frac{h^0(X,kL\otimes \CO_{kx_tZ})}{k^n/n!}=0.
\end{eqnarray*}
Then, $h^1(X,kL\otimes\CI_{Z}^{kx_t})=o(k^n)$. On the other hand, for  $j\in \IZ_{\gee 0}$, by the exact sequence
$$
0\rightarrow \CI_Z^{j}/\CI_Z^{j+1}\rightarrow \CO_X/\CI_Z^{j+1}\rightarrow \CO_X/\CI_Z^{j} \rightarrow 0.
$$
We have $H^i(X,kL\otimes \CI_Z^{j}/\CI_Z^{j+1}) = 0$ for $i\gee 1$ and $j\gee 1$ since $N_{Z/X}$ is trivial. So we get 
$$
H^i(X,kL\otimes \CO_X/\CI_Z^{j+1})\cong H^i(X,kL\otimes \CO_X/\CI_Z^{j})
$$ 
for all $i\gee 1$ and $j\gee 1$. Then, for $i\gee 2$,
\begin{eqnarray*}
    h^i(X,kL\otimes \CI_{Z}^{kx_t})=h^{i-1}(X,kL\otimes (\CO_X/\CI_Z^{kx_t}))= h^{i-1}(X,kL\otimes (\CO_X/\CI_Z))=0.
\end{eqnarray*}
And the higher direct images $R^i\pi_*\CO_{\hat{X}}=0$ for $i>0$, then by the Leray spectral sequence, we have $H^i(\hat{X},\pi^*(kL)-kx_tE) \cong H^i(X,kL\otimes \CI_Z^{xk_t})$ for $i\gee 0$. In particular, we get $h^i(\hat{X}, kL-kx_tE)=o(k^n)$ for $i\gee 1$. Then by \cite[Theorem A]{dFKL07}, we conclude that $\pi^*L-(\gamma-\eta)E$ is ample. This shows $\epsilon(L,Z)\gee \gamma$.
\end{proof}
\begin{proof}[Proof of Theorem \ref{thm-ntrivial}] 
 Since $X$ is assumed to be K-semistable, we know that $\delta(X)\ge 1$ by the valuative criterion (Theorem \ref{thm:beta}). Without loss of generality, we assume the codimension $r$ of $Z$ in $X$ is positive. By the proof of Proposition \ref{prop:Ntrivial}, the equality \eqref{eq-volprod} holds if and only if $\delta(X)=1$ and $\vol_{\hat{X}}(\pi^*(-K_X)-xE)=(\pi^*(-K_X)-xE)^n$ for all $x\in [0,r+1]$, so $\gamma_{Z}(-K_X)=r+1$ (see \eqref{eq-gamZ}). Then by Lemma \ref{lem-seshadri}, we have $\epsilon(-K_X,Z)=\gamma_Z(-K_X)=r+1$. By Proposition 
 \ref{prop-LZ}, we get $X\cong Z\times \IP^r$. 
\end{proof}
\begin{thm}\label{thm:toricsm}
     The second largest volume of $n$-dimensional K-semistable toric Fano manifold is $2n^n$ and the equality holds only if $X\cong \IP^1\times \IP^{n-1}$. 
\end{thm}
\begin{proof}
    When $X$ is a smooth toric Fano manifold, by the work of Chen-Fu-Hwang \cite[Corollary 2.5]{CFH14} which is partly based on the work of Araujo \cite{Ara06}, we know that there exists a submanifold $Z \cong \IP^{n-r}$ for $0\le r\le n-1$ that is contained in $X$ with a trivial normal bundle. By Proposition \ref{prop:Ntrivial} and $\delta(X)\gee 1$, we have $(-K_X)^n\lee (-K_{\IP^r\times Z})^n = (-K_{\IP^r\times \IP^{n-r}})^n=\binom{n}{r}\cdot (r+1)^r\cdot (n-r+1)^{n-r}=: c_r$. By the following lemma \ref{lem-cr}, for $1\le r\le n-1$,  $c_r\le 2n^n$ with equality holds if and only if $r=1$ or $r=n-1$ and in both cases $X\cong \IP^1\times \IP^{n-1}$. 
\end{proof}

\begin{lem}\label{lem-cr}
The sequence $c_{r}:=\binom{n}{r}(r+1)^r (n-r+1)^{n-r}$
with $0\le r\le n$ satisfies $c_r=c_{n-r}$ and $c_0=(n+1)^n>c_1=2n^n>c_2>\cdots> c_{\lfloor n/2\rfloor}$. 
\end{lem}
\begin{proof}
It is clear that $c_r=c_{n-r}$. We calculate:
\begin{eqnarray*}
   \frac{c_{r+1}}{c_r}&=&\left(\frac{r+2}{r+1}\right)^{r+1}\cdot \left(\frac{n-r}{n-r+1}\right)^{n-r}=a_{r+1}/a_{n-r}
\end{eqnarray*}
where $a_k=(\frac{k+1}{k})^k=(1+k^{-1})^k$ is an strict increasing sequence for $k\ge 1$ (that converges to $e$ as $k\rightarrow +\infty$). So we get $c_{r+1}/{c_r}<1$ when $r+1<n-r$ (or equivalently $2r+1<n)$.  The statement then follows easily. 
\end{proof}

\begin{prop}\label{prop:singulartoric}
    Let $X$ be a singular $n$-dimensional K-semistable toric $\IQ$-Fano variety. When $n=2$, then $(-K_X)^2\lee \frac{9}{2}$. When $n\gee 3$, then $(-K_X)^n \lee \frac{16}{27}(n+1)^n$. In particular, we have $(-K_X)^n< (-K_{\IP^1\times \IP^n})^n = 2n^n$ holds for all positive integer $n \gee 2$.
\end{prop}
\begin{proof}
    Suppose $x\in X$ lies in the singular locus and $(X,x)$ is a toric singularity. It was proved in \cite[Theorem 2]{MS24} that if $(X,x)$ is an $n$-dimensional $\IQ$-Gorenstein toric singularity, then 
    \begin{enumerate}
        \item If $n= 2$, then $\hvol(X,x)\lee 2$;
        \item If $n\geq 3$, then $\hvol(X,x) \lee \frac{16}{27}n^n$.
    \end{enumerate}
By Liu's local-to-global volume comparison (\cite{Liu18}), we have $(-K_X)^n \lee (\frac{n+1}{n})^n\cdot \hvol(X,x)$. Then when $n=2$, $(-K_X)^2 \lee (\frac{3}{2})^2\cdot 2 = \frac{9}{2}<2\cdot 2^2 = 8 = (-K_{\IP^1\times \IP^1})^2$. When $n\gee 3$, 
$$
(-K_X)^n \lee (\frac{n+1}{n})^n\cdot \frac{16}{27}n^n  =\frac{16}{27}\cdot (n+1)^n<2n^n = (-K_{\IP^1\times \IP^{n-1}})^n.
$$
\end{proof}
\begin{thm}\label{thm:vtoric}
 The second largest volume of $n$-dimensional K-semistable toric $\IQ$-Fano varieties is $2n^n$ and it is obtained only if $X\cong \IP^1\times \IP^{n-1}$.    
\end{thm}
\begin{proof}
    When $X$ is a smooth toric Fano manifold, this follows from Theorem \ref{thm:toricsm}. When $X$ is a strictly singular $\IQ$-Fano toric variety, this follows from Proposition \ref{prop:singulartoric}. When the equality $(-K_X)^n = 2n^n$ holds, from Proposition \ref{prop:singulartoric}, we know $X$ must be nonsingular. Then by Theorem $\ref{thm:toricsm}$ again, we know $X\cong \IP^1\times \IP^{n-1}$.
\end{proof}

\begin{rmk}\label{rem-sing}
    We remark that the singular version of Theorem \ref{thm-main} (that is, consider all $\IQ$-Fano varieties, not just smooth Fano manifolds) follows from ODP conjecture (Conjecture \ref{conj:ODP}). Let $X$ be a K-semistable $\IQ$-Fano variety and $x\in X$ be a singular point. Assume the ODP conjecture holds, by the same argument of Proposition \ref{prop:singulartoric}, we have $(-K_X)^n \lee (\frac{n+1}{n})^n\cdot \hvol(X,x)\lee (\frac{n+1}{n})^n\cdot 2(n-1)^n=(\frac{n+1}{n})^n\cdot (\frac{n-1}{n})^n\cdot 2n^n<2n^n$. Since the inequality is strict, it is expected that the equality of Theorem \ref{conj:2ndVol}  would hold only if $X$ is a smooth Fano manifold.
\end{rmk}
As mentioned in the Introduction, Theorem \ref{thm:vtoric} finds an application in Arakelov geometry on the sharp bound of height of arithmetic Fano varieties whose complexification is K-semistable. We refer our readers to \cite{AB24, Ber25} for more detailed explanation and possible arithmetic applications.
\begin{thm}
    Andreasson-Berman's conjecture holds for the toric canonical model of a toric Fano varieties. Namely, for an $n$-dimensional K-semistable toric Fano variety and let $\CX$ be the canonical model of $X$ over $\IZ$, then for any volume-normalized continous metric $\left\Vert \cdot\right\Vert$ on $-K_X$ with positive curvature current, we have $h_{\left\Vert \cdot\right\Vert}(\CX,-K_{\CX})\leq h_{\left\Vert \cdot\right\Vert_{\text{FS}}}(\IP^n_{\IZ},-K_{\IP^n_{\IZ}})$ where $-K_{\IP^n_{\IC}}$ is endowed with the volume normalized Fubini-Study metric. 
\end{thm}
\begin{proof}
    It follows directly from Theorem \ref{thm:vtoric} and \cite[Lemma 3.8]{AB24}.
\end{proof}

We conclude this section with another application in convex geometry. 
\begin{cor}\label{cor:polytope}
    Suppose $P\subseteq \IR^n$ is an $n$-dimensional reflexive lattice polytope with its barycenter at $0\in\IR^n$. Assume that $P$ is not unimodularly equivalent to $(n+1)$ times a standard simplex $(n+1)\Delta_n$, then the volume of $P$ with respect to the Lebesgue measure in $\IR^n$ satisfies $\vol_{\IR^n}(P)\lee 2n^n/n!$ and the equality holds if and only if $P$ is unimodularly equivalent to $[0,2]\times (n\Delta_{n-1})$.
\end{cor}
\begin{proof}
    Here $\Delta_n$ denote the standard $n$-dimensional simplex with vertices $0, e_1,\cdots,e_n$ where $e_i$ are standard lattice basis of $\IZ^n$ for $i=1,\cdots,n$. And two convex bodies $P$ and $Q$ in $\IR^n$ are called unimodularly equivalent if there exists an affine lattice automorphism of $\IZ^n$ mapping $P$ onto $Q$. Let $X_{P}$ be the projective toric variety associated to a reflexive polytope $P$. Then $X_P$ is a Gorenstein Fano variety (\cite[Theorem 8.3.4]{CLS11}). By the result of \cite{WZ04,SZ12,BB13}, we know that the barycenter of $P$ is zero if and only if $X_P$ is K-semistable. Since $P$ is not unimodularly equivalent to $(n+1)\Delta_n$, we have $X_P\ncong \IP^n$, then Theorem \ref{thm:vtoric} immediately implies $(-K_{X_P})^n\lee 2n^n$ and the equality holds only if $X_P\cong \IP^1\times \IP^{n-1}$. Then the reflexive polytope $P$ satisfies $\vol_{\IR^n}(P)\lee 2n^n/n!$ and $P\cong [0,2]\times (n\Delta_{n-1})$ up to unimodularly equivalent.
\end{proof}

\section{Volume estimates and minimal rational curves}\label{sec-main}

The calculation in the previous section is based on the test of K-stability via standard blow-ups. In this section, we prove Theorem \ref{thm-main} by carrying out a test of K-stability via a particular weighted blow-ups along minimal rational curves (see Example \eqref{ex-quadric} for the motivating example).
We also need to overcome the difficulty of possible singularities on the rational curve.

\subsection{Two divisorial valuations: $v_1$ and $v_2$ }\label{sub-coord}

Let $f: \IP^1\rightarrow X$ be a  
minimal rational curve. In other words, it is
an immersed rational curve that satisfies $f^*TX=\CO(2)\oplus\CO(1)^{\oplus (d-2)}\oplus \CO^{\oplus(n-d+1)}$ with $2\le d\le n+1$ and $d=f^*(-K_X)\cdot \IP^1=(-K_X)\cdot f(\IP^1)$. The (relative) normal bundle has the splitting: 
\begin{equation}\label{eq-split}
N_{f/X}=f^*TX/T\IP^1=\CO(1)^{\oplus (d-2)}\oplus\CO^{\oplus (n-d+1)}. 
\end{equation}
Correspondingly, we have the splitting of the conormal bundle:
\begin{equation*}
    N^\vee_{f/X}=\CO(-1)^{\oplus (d-2)}\oplus \CO^{\oplus(n-d+1)}. 
\end{equation*}
Set $Z=f(\IP^1)$ which is an irreducible curve with possibly a finite number of singularities $Z^{\mathrm{sing}}=\{p_1, \dots, p_k\}$ and $f: \IP^1\rightarrow Z$ is the normalization of $Z$. 
Set $f^{-1}(p_i)=\{t_i^j; 1\le j\le m_i \}$ to be the inverse image of the singularity $p_i$ under $f$ so that $f^{-1}(Z^{\mathrm{sing}})=\cup_i f^{-1}(p_i)$. 

For any point $t\in \IP^1$, we first construct a complex analytic coordinate chart near $f(t)\in X$ in the following way, which is adapted to the splitting \eqref{eq-split}. 
Choose a disk $B_t:=B_t(\epsilon)\subset \IP^1$ of radius $\epsilon>0$ (with respect to the round spherical metric) centered at $t$, such that $f: B_t\rightarrow X$ is an embedding. Since $f: \IP^1\rightarrow X$ is an immersion and $f^{-1}(Z^{\mathrm{sing}})$ is a finite set, it is easy to see that we can choose $\epsilon$ sufficiently small such that $B_t$ exists for all $t\in \IP^1$. Moreover, by choosing a uniform $0<\epsilon\ll 1$, we can choose an open neighborhood $U=U_t$ of $f(t)$ and complex analytic coordinates $\bz=\bz_t:=\{z_1, \dots, z_n\}$ on $U$ such that 
\begin{enumerate}
    \item $z_i(f(t))=0$ for $i=1,\dots, n$; $\bz(U_t)=\{|z_i|<\epsilon; i=1, \dots, n\}$; 
    \item 
    $f(B_t)\cap U_t=\{z_1=\cdots=z_{n-1}=0\}$;
    \item For any $1\le i\le d-2$, $f^*\dif z_i$ is a local generator of the $i$-th summand of $\CO(-1)^{\oplus(d-2)}\hookrightarrow N^{\vee}_{f/X}$ over $B_t$;
    \item For any $d-1\le j\le n-1$, $f^*\dif z_j$ is a local generator of the $j$-th summand of $\CO^{\oplus(n-d+1)}\hookrightarrow N^{\vee}_{f/X}$ over $B_t$.    
\end{enumerate}
We now define a divisorial valuation $v=\ord_E$ whose center over $X$ is $Z$. Pick any rational function $\psi$ on $X$.
We assume that $\psi$ is defined on a Zariski open set $W$ such that $W\cap Z\neq \emptyset$.
Pick any point $t\in f^{-1}(W\cap Z)$. 
The restriction $\psi$ to $U=U_t$ becomes holomorphic near $0\in U$. Expand the restriction $f|U$ near $0\in U$:
\begin{equation}\label{eq-psiex}
    \psi(\bz)=\sum_{I,J} {b_{IJ}}(z_n) \bz'^I \bz''^J
\end{equation}
where $\bz'=\{z_1, \dots, z_{d-2}\}$ (resp. $\bz''=\{z_{d-1},\cdots, z_{n-1}\}$), and $\bz'^I=z_1^{i_1}\cdots z_{d-2}^{i_{d-2}}$ for $I=(i_1, \dots, i_{d-2})\in \IN^{d-2}$ with $|I|=i_1+\cdots+i_{d-2}$ (resp. $\bz''^{J}=z_{d-1}^{j_1}\cdots z_{n-1}^{j_{n-d+1}}$ for $J=(j_1,\cdots,j_{n-d+1})\in \IN^{n-d+1}$ with $|J|=j_1+\cdots+j_{n-d+1}$). 
We then define: for $\ell=1$ or $\ell=2$
\begin{equation}\label{eq-vpsi}
    v_\ell(\psi)=\min\{|I|+\ell |J|; \psi=\sum_{I,J} {b_{IJ}}\bz'^I \bz''^J, b_{IJ}=b_{IJ}(z_n)\not\equiv 0\}.
\end{equation}
We claim that this is independent of $t\in f^{-1}(W\cap Z)$. Since $Z\setminus W$ is a finite set, $W\cap Z$ is connected. So we just need to verify the claim locally. Assume that $t'$ is near $t$ so that $x'=f(t')$ is near $x=f(t)$ and we
can find a sequence of points $\{t_i; i=1, \dots, N\}$ and small discs $B_{t_i}(\epsilon_1)\subset B_{t_i}(\epsilon)=B_{t_i}$ with $0<\epsilon_1\ll \epsilon$ such that $t_1=t,\,t_N=t'$ and $B_{t_i}(\epsilon_1)\cap B_{t_{i+1}}(\epsilon_1)\supset B_{t_i}(\epsilon_1/3)\cup B_{t_{i+1}}(\epsilon_1/3)$. In particular, this implies
such that $U_{t_i}\cap U_{t_{i+1}}\neq \emptyset$. 
For each $1\le i\le N-1$, the transition function on $U_{t_i}\cap U_{t_{i+1}}$ from $(U_{t_i}, \bz=\bz_{t_i})$ and $(U_{t_{i+1}}, \hat{\bz}:=\bz_{t_{i+1}})$ is of the form:
\begin{eqnarray}\label{eq-trfct}
    z_k&=&a_k(\hat{z}_n) \hat{z}_k+R_2, 1\le k\le n-1, \quad z_n=g_n(\hat{z}_n)+R_1 
\end{eqnarray}
where $a_k(\hat{z}_n)$ is a non-zero holomorphic transition function of $\CO_{\IP^1}(-1)$ for $1\le k\le d-2$ and of $\CO_{\IP^1}$ for $d-1\le k\le n-1$, $g_n(\hat{z}_n)$ is an invertible holomorphic function and $R_2$ (resp. $R_1$) are holomorphic functions whose expansion consists only of terms of degree at least 2 (resp. at least 1) in $\hat{z}_1, \dots, \hat{z}_{n-1}$. 
Assume $\psi|_{U_{t_i}}$ has the expansion \eqref{eq-psiex} centered at $t_i\in U_{t_i}$ with all terms of weighted degree $|I|+\ell|J|$ at least $y$, i.e. the right-hand-side of \eqref{eq-vpsi} is equal to $y$. 
Then
\begin{eqnarray*}
    \psi(\bz)
    &=&\sum_{|I|+\ell |J|\ge y} b_{IJ}(g_n(\hat{z}_n)+R_1) (\mathbf{a}'(\hat{z}_n)\hat{\bz}'+R_2)^I (\mathbf{a}''(\hat{z}_n) \hat{\bz}''+R_2)^J\\
    &=&\sum_{|I|+\ell |J|=y} b_{IJ}(g_n(\hat{z}_n))(\mathbf{a}'(\hat{z}_n)\hat{\bz}')^I (\mathbf{a}''(\hat{z}_n)\hat{\bz}'')^J+\sum_{|I|+\ell|J|=y}b_{IJ}(g_n(\hat{z}_n))(\mathbf{a}'(\hat{z}_n)\hat{\bz}')^I  R'^J_2+R^w_{y+1}\\
   &=& \sum_{K,L}\hat{b}_{KL}(\hat{z}_n) \hat{\bz}'^K\hat{\bz}''^L
\end{eqnarray*}
where $(\mathbf{a}'\hat{\mathbf{z}}')^I=\prod_{k=1}^{d-2} a_k^{i_k}z_k^{i_k}$,
$(\mathbf{a}''\hat{\mathbf{z}}'')^J=\prod_{k=d-1}^{n-1} a_k^{j_k}z_k^{j_k}$, $R'_2$ consists of terms of degree at least $2$ in $(\hat{z}_1,\dots, \hat{z}_{d-2})$, $R^w_{y+1}$ consists of terms of weighted degree $|I|+\ell |J|$ at least $y+1$ and the last expression is the expansion of $f|U_{t_{i+1}}$ near the center $t_{i+1}\in U_{t_{i+1}}$. 
It is then easy to see that for $\ell=1$ or $\ell=2$,
\begin{equation*}
\min\{|K|+\ell |L|; \psi=\sum_{K, L}\hat{b}_{KL}(\hat{z}_n) \hat{\bz}'^K \hat{\bz}''^L, \hat{b}_{KL}\not\equiv 0\}
\end{equation*}
is at least the quantity on the right-hand-side of \eqref{eq-vpsi}. Since we can switch the role of $t_i$ and $t_{i+1}$, we know that the two quantities are actually equal. By passing from $t_1=t$ to $t_N=t'$, we conclude that for $\ell=1,2$, $v_\ell(f)$ does not depend on $t$ and $t'$. 

\begin{rmk}
    Definition of quasi-monomial valuations by using analytic expansions under complex analytic coordinates has also recently appeared in \cite[2.3.1]{MP24}. In particular, a similar discussion about the independence of local coordinate charts is contained in \cite[Proposition 2.3.4]{MP24}.
\end{rmk}

 When restricted to each $U=U_{t}$, the valuation $v_\ell$ corresponds to the exceptional divisor $E$ of the weighted blowup of $U_t$ along $f(B_t)\cap U_t$ with weights 
 \begin{equation*}
     (1^{\oplus (d-2)}, \ell^{\oplus (n-d+1)})=
     \left\{
     \begin{array}{ll}
     (\underbrace{1, \dots, 1}_{n-1})& \text{for } \ell=1, \\
     (\underbrace{1, \dots, 1}_{d-2}, \underbrace{2, \cdots, 2}_{n-d+1}) & \text{for } \ell=2
     \end{array}
     \right.
 \end{equation*}(See \cite{QR22} for more a general set-up of weighted blowups). In fact, for $\ell=1$ the valuation $v_1$ corresponds to the symbolic blowup of $X$ along $Z$.  

Fix $\ell\in \{1,2\}$ and set $\CI_{\lambda}=\CI_{\lambda}(v_\ell)=\{f\in \CO_X; v_\ell(f)\ge \lambda\}$. Choose $\lambda=x k$ for $x\in \IQ_{>0}$. Then similar to the proof of Proposition \ref{prop:Ntrivial}, we have the exact sequence:
\begin{eqnarray*}
    0\rightarrow H^0(X, L^k\otimes \CI_{xk})\rightarrow H^0(X, L^k)\rightarrow H^0({X}, L^k\otimes \CO_X/\CI_{xk})\rightarrow \cdots,
\end{eqnarray*}
which implies: 
$$\vol(L-xE)\ge L^n-\mathop{\limsup}_{k\rightarrow+\infty} \frac{h^0({X}, L^k\otimes \CO_X/\CI_{xk})}{k^n/n!}.$$
By using the exact sequence:
\begin{equation*}
    0\rightarrow \CI_y/\CI_{y+1}\rightarrow \CO_X/\CI_{y+1}\rightarrow \CO_X/\CI_y\rightarrow 0,
\end{equation*}
we get inductively the estimates:
\begin{eqnarray}\label{eq-tels2}
h^0(X, L^k\otimes \CO_X/\CI_{xk})
&\le&\sum_{y=0}^{xk-1}h^0(X, L^k\otimes \CI_{y}/\CI_{y+1}).
\end{eqnarray}

\subsection{Singular minimal rational curves with trivial normal bundles}
Our goal in this subsection is to prove the following gap result when there is an immersed rational curve with a trivial bundle and a singular image set. 
\begin{prop}\label{prop-deg2}
    Let $X$ be a K-semistable Fano manifold of dimension $n\ge 3$. 
    Assume that $f: \IP^1\rightarrow X$ is an immersed rational curve with a non-empty subset of singular points such that $f^*TX=\CO(2)\oplus \CO^{n-1}$. Then we have an estimate of the volume $(-K_X)^n< 2n^n$.
    As a consequence, if there exists a minimal rational curve with a trivial normal bundle on $X$ and $(-K_X)^n=2n^n$, then the minimal rational curve has no singular points and $X$ must be isomorphic to $\IP^1\times \IP^{n-1}$.  
\end{prop}

The rest of this subsection is devoted to proving the above proposition. The last statement follows from the strict volume inequality and Theorem \ref{thm-ntrivial}. So we just need to prove the inequality $(-K_X)^n< 2n^n$ when $f$ is not an embedding. We use the same notation as the previous subsection but now specialize to the case when $\ell=1$ and $d=2$.

For any $[s]\in H^0(X, L^k\otimes \CI_y/\CI_{y+1})$, choose a local generator $\sigma_t$ of $L^k$ over $U_t$, the section $[s]$ is represented uniquely over $U_t$ by a holomorphic function with the following expression (note that, since $d=2$, we get $d-2=0$ so that $(z_1,\dots, z_{n-1}, z_n)=(\bz'', z_n)$): 
\begin{eqnarray*}
    \frac{s_t}{\sigma_t}&=&\sum_{|J|=y}b^t_{J}(z_n) \bz''^J
\end{eqnarray*}
which we call the initial term representation of $s$ over $U_t$. 
Note that the superscript $t$ in $b^t_J$ means the dependence on $t$ (and not the transpose). 
Set
\begin{equation*}    \gamma_t([s])=\sigma_t\cdot \sum_{|J|=y}b^t_{J}(z_n) (\dif\bz'')^{\otimes J}
\end{equation*}
which defines a holomorphic section of $f^*L^k\otimes \mathrm{Sym}^{y}(\CO^{\oplus (n-1)})$ over $B_t$. It is easy to check that $\gamma_t([s])$ patches together to define a global section $$\gamma([s])\in H^0(\IP^1, f^*L^k\otimes \mathrm{Sym}^y (\CO^{\oplus n-1}))= H^0(\IP^1, f^*L^k)\otimes \IC^{\binom{n+y-2}{y}}.
$$
In other words, we know that for each $J=(j_1,\dots, j_{n-1})$ with $|J|=j_1+\cdots+j_{n-1}=y$, 
$b_J^t(z_n)(\dif\bz'')^{\otimes J}$ patches together to become a section $s_J$ in $H^0(\IP^1, f^*L^k)$. 

Now assume that $t\neq t'\in \IP^1$ satisfy $f(t)=f(t')=x\in f(\IP^1)=Z\subset X$. 
Let $C=f(B_t)$ and $C'=f(B_{t'})$ be the two distinct locally smooth branches of $Z$ that pass
through $x$. Let $(U=U_t, \bz=\bz_t)$ and $(\hat{U}=U_{t'}, \hat{\bz}=z_{t'})$ be coordinate charts centered at $t$ and $t'$ respectively in Subsection \ref{sub-coord}. 
In other words, $z=\{z_1,\dots, z_n\}$ (resp. $z'=\{z'_1,\dots, z'_n\})$ satisfies $z_k(f(t))=0$ (resp. $z'_k(f(t'))=0)$ for any $1\le k\le n$; $C=\{z_1=\cdots z_{n-1}=0\}$ (resp. $C'=\{z'_1=\cdots=z'_{n-1}=0\}$) near $f(t)$ (resp. $f(t')$). Moreover for $1\le j\le n-1$, $f^*dz_j$ (resp. $f^* dz'_j$),  is a local generator for the $j$-th summand of $\CO^{n-1}=N^\vee_{f/X}$ near $f(t)$ (resp. $f(t')$).  

For each $1\le k\le n$, set $m_k=\mathrm{ord}_{\hat{z}_n=0}z_k(0,\dots, 0, \hat{z}_n)\ge 0$, then we can write:
\begin{eqnarray*}    z_k&=&g_{k1}(\hat{z}_n)\hat{z}_1+\cdots+g_{k,{n-1}}(\hat{z}_n)\hat{z}_{n-1}+\hat{z}_n^{m_k} h_k(\hat{z}_n)+R_2, 
\end{eqnarray*}
where $g_{kj}(\hat{z}_n)$ and $h_k(\hat{z}_n)$ are holomorphic functions of $\hat{z}_n$ with $h_k(0)\neq 0$ if $m_k>0$ or $h_k(\hat{z}_n)\equiv 0$ if $m_k=0$, and $R_2$ consists of terms of degree at least $2$ in $\hat{z}_1, \dots, \hat{z}_{n-1}$. 
Without the loss of generality, assume that $m_1\ge m_2\ge \cdots \ge m_{n-1}\ge 0$. Then we must have $m_1>0$. Otherwise, we have $C'=\{\hat{z}_1=\cdots=\hat{z}_{n-1}=0\}\subseteq C=\{z_1=\cdots=z_{n-1}=0\}$ which is not possible. Assume that $1\le q\le n-1$ satisfies $m_{q}>0$ and $m_{q+1}=0$. If $q<n-1$, then for any $q+1\le k\le n-1$, we have $\dif z_k|_{T_xX}=\sum_{j=1}^{n-1} g_{kj}(0) \dif \hat{z}_j|_{T_xX}$ which implies that the matrix $(g_{kj}(0))^{q+1\le k\le n-1}_{1\le j\le n-1}$ has rank equal to $n-1-q$. 
For the section $s_J$ obtained from above, we get: 
\begin{eqnarray}\label{eq-bJtr}
    &&b^t_J(z_n) z_1^{j_1}z_2^{j_2}\cdots z_{n-1}^{j_{n-1}}=b^{t}_J(z_n) (\prod_{k=1}^q \hat{z}_n^{m_k j_k} h_k(\hat{z}_n)^{j_k}+R_1)\cdot \prod_{k=q+1}^{n-1}
    \left(\sum_{j=1}^{n-1} g_{kj}(\hat{z}_n) \hat{z}_j+R_2\right)^{j_k}\\
    &=&z_n^{\beta} c_J(z_n)\prod_{k=1}^q 
    (\hat{z}_n^{m_kj_k}h_k(\hat{z}_n)^{j_k}+R_1) 
    \left(\prod_{k=q+1}^n \left(\sum_{j=1}^{n-1}g_{kj}(\hat{z}_n)\hat{z}_j\right)^{j_k}+R_{{j_{q-1}}+\cdots+j_{n-1}+1}\right)\nonumber 
\end{eqnarray}
where we denote $\beta=\ord_{z_n=0}(b^t_J(z_n))\ge 0$ and $b^t_J(z_n)=z_n^\beta c_J(z_n)$ with $c_J(z_n)\not\equiv 0$. 
Note that if $q=n-1$, then the last factor should be considered as $1$. 
Also recall that for any $\theta\in \mathbb{N}_{>0}$, 
we always use $R_{\theta}$ to denote a holomorphic function of $(\hat{z}_1, \dots, \hat{z}_{n-1}, \hat{z}_n)$ whose expansion consists of terms of degree at least $\theta$ in $(\hat{z}_1,\cdots, \hat{z}_{n-1})$. For the purpose of the current argument, for $\theta\ge 0$ we now use $A_\theta$ to denote any non-zero holomorphic function of $(\hat{z}_1, \dots, \hat{z}_{n-1}, \hat{z}_n)$ whose expansion is a nonzero polynomial of degree equal to $\theta$ in $(\hat{z}_1,\dots, \hat{z}_{n-1})$. Then the right-hand-side of \eqref{eq-bJtr} has the form:
\begin{eqnarray*}\label{eq-bJ2}
    z_n^\beta (A_0+R_1)(A_0+R_1)(A_{j_{q+1}+\cdots+j_{n-1}}+R_{j_{q+1}+\cdots+j_{n-1}+1})
    &=& z_n^\beta (A_{j_{q+1}+\cdots+j_{n-1}}+R_{j_{q+1}+\cdots+j_{n-1}+1})
\end{eqnarray*}
If $q=n-1$, then the last factor should be considered as $A_0$. 

Now observe that the right-hand-side of the above formula \eqref{eq-bJtr} should be of the form $$b_{J}^{t'}(\hat{z}_n)\hat{z}_1^{j_1}\hat{z}_2^{j_2}\cdots \hat{z}_{n-1}^{j_{n-1}}+R_{|J|+1}=A_{|J|}+R_{|J|+1}$$
where $|J|=j_1+\cdots+j_{n-1}$. 
This is due to the fact that by construction $f^* \dif z_j=f^* d\hat{z}_j$ on $U_t\cap U_{t'}$ which is the restriction to $U_t\cap U_{t'}$ of a global non-zero section of the $j$-th factor of $\CO^{\oplus (n-1)}=N_{f/X}^{\vee}$.
Considering $z_n$ as a function of $(\hat{z}_1, \cdots, \hat{z}_n)$ that satisfies $\partial z_n/\partial z_i(0,\dots, 0)\neq 0$ for some $i\in \{1,\dots, n\}$, this forces $z_n=A_1+R_2$ (i.e. $\partial z_n/\partial z_k(0,\dots, 0)\neq 0$ for some $k\in \{1,\dots, n-1\}$) and  $\beta=|J|-(j_{q+1}+\cdots+j_{n-1})=j_1+\cdots+j_{q}\ge j_1$. 
So we conclude that $s_J\in H^0(\IP^1, f^*L^k\otimes \CO(-1)^{\otimes j_1}\otimes \CO^{\otimes j_2}\otimes \cdots \otimes\CO^{\otimes j_{n-1}})$, and get the estimate:
\begin{eqnarray}\label{eq-tneg}
    h^0(C, L^k\otimes \CI_y/\CI_{y+1})\le h^0(\IP^1, \CO(2k)\otimes \mathrm{Sym}^y(\CO(-1)\oplus \CO^{\oplus(n-2)})). 
\end{eqnarray}
By \eqref{eq-tels2} and \eqref{eq-tneg}, we get:
\begin{eqnarray}\label{eq-tels1}
    \sum_{y=0}^{xk-1} h^0(C, L^k\otimes \CI_y/\CI_{y+1})
    &\le& \sum_{y=0}^{xk-1} h^0(\IP^1, \CO(2k)\otimes \sym^y(\CO(-1)\oplus \CO^{\oplus(n-2)}))\nonumber \\
    &\le&  \sum_{y=0}^{xk-1}\sum_{i=0}^{\min\{y,2k\}}(2k-i+1)a_{y,i}
\end{eqnarray}
where $a_{y,i}=\binom{n-3+y-i}{n-3}$.


If $x\lee 2$, then by using the explicit expression of binomial coefficient, the right-hand-side of \eqref{eq-tels1} is equal to $b_1 k^n/n!+O(k^{n-1})$ where $b_1$ is given by the following integral (see \eqref{eq-phiint} for a derivation of a more complicated formula):
\begin{eqnarray*}
    b_1&=&n!\int_0^x \dif t\int_0^t  (2-z)\frac{(t-z)^{n-3}}{(n-3)!}\dif z
    = \frac{n!}{(n-2)!}\int_0^x (2-z)(x-z)^{n-2}\dif z\\
    &=&2n x^{n-1}- x^n<2n x^{n-1}.
\end{eqnarray*}
When $x> 2$, the right-hand-side of \eqref{eq-tels1} splits into the sum:
$\sum_{y=0}^{2k}\sum_{i=0}^{y}+\sum_{y=2k+1}^{xk-1}\sum_{i=0}^{2k}$ and the right-hand-side of \eqref{eq-tels1} is equal to $b'_1 k^n/n!+O(k^{n-1})$ with 
\begin{eqnarray*}
b'_1&=&n!\int_0^2 \dif t\int_0^t  (2-z)\frac{(t-z)^{n-3}}{(n-3)!}\dif z+n!\int_2^x \dif t \int_0^2 (2-z)\frac{(t-z)^{n-3}}{(n-3)!}\dif z\\
&=&\frac{n!}{(n-2)!}\int_0^2(2-z)(x-z)^{n-2}\dif z\\
&=&2 nx^{n-1}-x^n+(x-2)^n<(2n-1)x^{n-1}+2^{n-1}.
\end{eqnarray*}
Set $\psi(x)=2n x^{n-1}$ and 
\begin{equation}
   {\phi}(x)=
   \left\{\begin{array}{ll}
     2n x^{n-1} &  \text{if }\; 0\le x\le 2\\
    (2n-1)x^{n-1}+2^{n-1} & \text{if }\; x\ge 2, 
\end{array}
\right.
\end{equation}
so that $\phi(x)<\psi(x)$ with strict inequality when $x>2$. Moreover it is easy to check that $(\phi^{-1})'(V)\ge (\psi^{-1})'(V)$ and the strict inequality holds when $V>\phi(2)=n 2^{n}=\psi(2)$. So we can apply Lemma 
\ref{lem-cal} to conclude that $F^{\phi}(V)\ge F^{\psi}(V)$ and the strict inequality holds for $V> n 2^{n}$. By Example \ref{ex-2nn}, $F^\psi(n 2^{n})=2\frac{n-1}{n}$, which is less than $A_X(E)=n-1$ if $n\ge 3$. So we can apply 
Corollary \ref{cor-calbd} and \ref{lem-cal}.(2) to get $(-K_X)^n\le (F^{\phi})^{-1}(n-1)<(F^{\psi})^{-1}(n-1)=2n^n$. This completes the proof of Proposition \ref{prop-deg2}. 

\subsection{Minimal rational curves with degree $d\ge 3$}
In this section, we are going to prove the following result which resolves a large part of Theorem \ref{thm-main}. Recall the invariant $l_X$ from \eqref{eq-lX}, whose definition is the smallest degree of minimal rational curves on $X$. 
\begin{prop}\label{prop-3ton}
    Let $X$ be a K-semistable Fano manifold with $3\le l_X\le n-1$. Then $\vol(X)<2n^n$. 
\end{prop}
 We use the notation from subsection \ref{sub-coord} and specialize to the case when $\ell=2$ and $3\le d\le n+1$. 

For any $[s]\in H^0(X, L^k\otimes \CI_y/\CI_{y+1})$, choose a local generator $\sigma_t$ of $L^k$ over $U_t$, the section $[s]$ is represented uniquely over $U_t$ by a holomorphic function with the following expression:  
\begin{eqnarray*}
    \frac{s_t}{\sigma_t}&=&\sum_{|I|+2|J|=y}b^t_{IJ}(z_n)\bz'^I\bz''^J=\sum_{m=0}^{\lfloor y/2\rfloor}
    \sum_{\substack{|I|+|J|=y-\lfloor y/2\rfloor+m\\
    |I|=y-2\lfloor y/2\rfloor+2m\\ |J|=\lfloor y/2\rfloor-m
    }}b_{IJ}^t(z_n) \bz'^I \bz''^J
\end{eqnarray*}
which we call the initial term representation of $s$ over $U_t$. We set
\begin{equation*}
    \gamma_t([s]):=\sigma_t\cdot \sum_{\substack{|I|=y-2\lfloor y/2\rfloor \\
    |J|=\lfloor y/2 \rfloor\\   
    |I|+|J|=y-\lfloor y/2\rfloor}} b^t_{IJ}(z_n)(\dif\bz')^{\otimes I} (\dif \bz'')^{\otimes J}
\end{equation*}
which defines a local section of $f^*L^k\otimes \sym^{y-2 \lfloor y/2\rfloor}(\CO(-1)^{\oplus (d-2)})\otimes\sym^{\lfloor y/2\rfloor}(\CO^{\oplus(n-d+1)}))$ over $B_t$. 
We claim that this section is well-defined over the whole $\IP^1$, i.e. it is independent of $t\in \IP^1$. For any other $t'\in \IP^1$, choose a sequence $\{t_i; i=1, \dots, N\}$ and $B_{t_i}(\epsilon')$ as before. 
We use the transition function \eqref{eq-trfct} to get
\begin{eqnarray*}
    \frac{s_t}{\sigma_{t}}&=&
    \sum_{|I|+2|J|=y} b^{t}_{IJ}(g_n(\hat{z}_n)+R_1) (\mathbf{a}'(z_n)\hat{\bz}'+R_2)^I (\mathbf{a}''(z_n) \hat{\bz}''+R_2)^J\\
    &=&\sum_{|I|+2|J|=y} b^t_{IJ}(g_n(\hat{z}_n))(\mathbf{a}'(z_n)\hat{\bz}')^I (\mathbf{a}''(z_n)\hat{\bz}'')^J+\sum_{|I|+2|J|=y}b^t_{IJ}(g_n(\hat{z}_n))(\mathbf{a}'(z_n)\hat{\bz})'^I  R'^J_2+R^w_{y+1}\\
    &=:& \frac{s_{t'}}{\sigma_{t'}}\frac{\sigma_{t'}}{\sigma_t}+\tilde{R}^w_{y+1}
\end{eqnarray*}
where $R'_2$ is quadratic in $\hat{z}_1,\dots, \hat{z}_{d-1}$, $R^w_{y+1}$ consists of terms of weighted degree $2|I|+J\ge y+1$ and $\tilde{R}^w_{y+1}$ consists of terms with $2|I|+J\ge y$ and $|I|+|J|>y-2\lfloor y/2\rfloor$.  
As a consequence, we get:
   \begin{eqnarray*}
    \gamma_{t'}([s])
    &=&\sigma_t\cdot \sum_{\substack{|I|=y-2\lfloor y/2\rfloor\\ |J|=\lfloor y/2\rfloor}}b^t_{IJ}(g_n(\hat{z}_n))(a'(z_n)\dif\hat{\bz}')^{\otimes I} (a''(z_n)\hat{\bz}'')^{\otimes J}
\end{eqnarray*}
which coincides with $\gamma_t([s])$ on $B_t\cap B_{t'}$. 

So we get a linear map of vector spaces:
\begin{equation*}
    \gamma: H^0(X, L^k\otimes \CI_y/\CI_{y+1})\rightarrow H^0(\IP^1, f^*L^k\otimes\sym^{y-2 \lfloor y/2 \rfloor}(\CO(-1)^{\oplus (d-2)})\otimes \sym^{\lfloor y/2\rfloor}(\CO^{\oplus (n-d+1)})),
\end{equation*}
whose kernel consists of elements $[s]\in H^0(X, L^k\otimes \CI_y/\CI_{y+1})$ that has the initial term representation given by ($m\ge 1$ now):
\begin{equation*}
    \frac{s_t}{\sigma_t}=\sum_{m=1}^{\lfloor y/2\rfloor}\quad  \sum_{\substack{|I|+|J|=y-\lfloor y/2 \rfloor+m \\ |I|=y-2 \lfloor y/2\rfloor+2 m\\ |J|=\lfloor y/2\rfloor-m}} b_{IJ}(z_n) \bz'^I \bz''^J,
\end{equation*}
which by the same argument as above induces a section in $H^0(\IP^1, f^*L^k\otimes \sym^{y-2\lfloor y/2\rfloor +2}(\CO(-1)^{\oplus(d-2)})\otimes \sym^{\lfloor y/2 \rfloor-1}(\CO^{\oplus (n-d+1)}))$. 
So we inductively get the estimate:
\begin{eqnarray}
    && h^0(X, L^k\otimes \CI_y/\CI_{y+1})\\
    &\le& 
    \sum_{m=0}^{\lfloor y/2\rfloor }h^0(\IP^1, \CO(kd)\otimes \sym^{y-2 \lfloor y/2\rfloor+2m}(\CO(-1)^{\oplus (d-2)})\otimes \sym^{\lfloor y/2\rfloor -m}(\CO^{\oplus (n-d+1)}))\nonumber \\
    &\le&\sum_{m=0}^{\lfloor y/2\rfloor}
    \binom{d-3+2m+\epsilon_y}{d-3}\binom{n-d+\lfloor y/2\rfloor-m}{n-d}h^0(\IP^1, \CO(kd-2m))  \label{eq-west}
\end{eqnarray}
where $\epsilon_y=1$ if $y$ is odd and $\epsilon_y=0$ if $y$ is even. 
\begin{rmk}\label{rmk-Zhuang}
    In the first version of this preprint, we defined valuations $v_\ell$ for any $\ell\ge 1 \in \IN$ and claimed that the above estimate is true for any $\ell\ge 1\in \IN$. Ziquan Zhuang pointed out that for $\ell \ge 3$, this is not true. The reason is that while $v_1$ and $v_2$ are well-defined globally, the valuation $v_\ell$ for $\ell \ge 3$ depends on a collection of coordinate charts that only cover a dense open subset of $f(\IP^1)$ (the complement of finite many points), which causes the failure of the above estimate. For $v_1$ and $v_2$ the coordinates charts we constructed cover the whole $f(\IP^1)$ and there is no such an issue. By a similar argument in Subsection \ref{sub-coord}, a sufficient condition for $v_\ell$ to be defined globally is that the transition functions between coordinate charts covering the whole $f(\IP^1)$ are of the form:
    \begin{equation*}
        z_k=a_k(\hat{z}_n)\hat{z}_k+R_{\ell}, 1\le k\le n-1; z_n=g_n(\hat{z}_n)+R_{\ell-1} 
    \end{equation*}
where $R_{\ell}$ (resp. $R_{\ell-1}$) consists of terms of degree at least $\ell$ (resp. at least $\ell-1$) in $\hat{z}_1,\dots, \hat{z}_{n-1}$. 
    When the minimal rational curve is embedded, this condition is satisfied exactly when the rational curve is $(\ell-2)$-comfortably embedded in the sense of Abate-Bracci-Tovena (see \cite[Remark 3.3]{ABT09}). The latter condition is a refinement and a slight weakening of Grauert's $(\ell-1)$-linearizability condition. From this point of view, the valuation $v_1$, $v_2$ are always globally defined since any minimal rational curve is always $1$-linearizable, and in particular (trivially) $0$-comfortably embedded.

\end{rmk}
 By the similar (and simpler) argument as above, we get the similar estimate for the valuation $v_1$ (i.e $v_\ell$ from section \ref{sub-coord} with $\ell=1$):
    \begin{eqnarray*}
h^0(X, L^k\otimes \CI_y(v_1)/\CI_{y+1}(v_1))
    &\le & \sum_{m=0}^{y} h^0(\IP^1, \CO(kd)\otimes \sym^{m}(\CO(-1)^{\oplus (d-2)})\otimes \sym^{y-m}(\CO^{\oplus (n-d+1)})).
    \end{eqnarray*}
By \eqref{eq-tels2} and \eqref{eq-west}, for each $v_\ell$ with $\ell\in \{1,2\}$, we get that $h^0(X, L^k\otimes \CO_X/\CI_{xk}(v_\ell))$ is bounded from above by: \footnote{Our main interest will be the case $\ell=2$. We include the parameter $\ell=1$ for comparison later. }
\begin{eqnarray*}
I_k(x):=\sum_{y=0}^{xk-1}\sum_{m=0}^{\lfloor y/\ell\rfloor}a_{y,m}h^0(\IP^1, \CO(kd-\ell m)),
\end{eqnarray*}
where the coefficients 
\begin{eqnarray*}
    a_{y,m}=\binom{d-3+\ell m+\epsilon_y}{d-3}\binom{n-d+\lfloor y/\ell\rfloor-m}{n-d}, 
\end{eqnarray*}
where $\epsilon_y=1$ if $y$ is odd and $\epsilon_y=0$ if $y$ is even. 
We will calculate the leading coefficient $b_1=\lim_{k\rightarrow +\infty} (I_k(x)\cdot n!)/k^n$. 
First, we assume that $x\le d$. Then we can expand:
\begin{eqnarray*}
I_k(x)&=&\sum_{y=0}^{xk-1}\sum_{m=0}^{\lfloor y/\ell\rfloor}\frac{(\ell m)^{d-3}+O((\ell m)^{d-4})}{(d-3)!}\frac{(\lfloor y/\ell \rfloor -m)^{n-d}+O((\lfloor y/\ell\rfloor -m)^{n-d-1})}{(n-d)!}(kd-\ell m+1)\\
&=&\left[\frac{k^n}{(d-3)!(n-d)!}\sum_{y=0}^{xk-1}\frac{1}{k}\sum_{m=0}^{
\lfloor y/\ell\rfloor}\frac{1}{k} (\ell \frac{m}{k})^{d-3}(\ell^{-1}\frac{y}{k}-\frac{m}{k})^{n-d}(d-\ell \frac{m}{k})\right]+O(k^{n-1}).
\end{eqnarray*}
As $k\rightarrow+\infty$, we see that $I_k(x)=b_1 k^n/n!+O(k^{n-1})$ with $b_1$ given by the integral:
\begin{eqnarray}\label{eq-phiint}
    &&\frac{n!}{(d-3)!(n-d)!}\int_0^x \dif t \int_0^{t/\ell} (\ell s)^{d-3}(\ell^{-1}t-s)^{n-d}(d-\ell s) \dif s\nonumber \\
    &=&\frac{n!}{(d-3)!(n-d)!}\ell^{-(n-d+1)}\int^{x}_0\dif t \int_0^{t}  z^{d-3}(d-z)(t-z)^{n-d} \,\dif z\nonumber \\
    &=&\ell^{-(n-d+1)} \frac{n!}{(d-3)!(n-d+1)!}\int_0^x z^{d-3} (d-z)(x-z)^{n-d+1}\, \dif z\nonumber \\
    &=& \ell^{-(n-d+1)} x^{n-1}\left(dn-(d-2)x\right)\eqqcolon\phi(x).
\end{eqnarray}
It is easy to see that $\phi(x)$ is an increasing function when $x\in [0, \frac{(n-1)d}{(d-2)}]$ with $\phi(d)=\ell^{-(n-d+1)}d^{n}(n-d+2)$. In particular, $\phi(x)$ is strictly increasing when $x \in [0,d]$ since $d\le n+1$.

Next, consider the case when $x\gee d$. Then we need to split the sum in \eqref{eq-tels2} into two parts:
\begin{eqnarray*}
    &&\sum_{y=0}^{dk}\sum_{i=0}^{\lfloor y/\ell\rfloor} a_{y,i}h^0(\IP^1, \CO(kd-\ell i))+\sum_{y=dk+1}^{xk-1}\sum_{i=0}^{\lfloor dk/\ell \rfloor}a_{y,i}h^0(\IP^1, \CO(kd-i\ell)).
\end{eqnarray*}
Similar calculation as above shows that this sum equals $b'_1 k^n/n!+O(k^{n-1})$ with $b'_1$ equal to $C\cdot \mathbf{I}$ where $C= \frac{n!}{(d-3)!(n-d)!}\ell^{-(n-d+1)}$ and $\mathbf{I}$ is equal to:
\begin{eqnarray*}
    &&\int_0^d \dif t\int_0^t z^{d-3}(t-z)^{n-d}(d-z)\dif z+\int_d^x \dif t\int_0^d z^{d-3}(t-z)^{n-d}(d-z)\dif z\\
    &=& C^{-1}\phi(d)+\int_0^d z^{d-3} \frac{1}{n-d+1}\left((x-z)^{n-d+1}-(d-z)^{n-d+1}\right)(d-z)\dif z\\
    &=&\frac{1}{n-d+1}\int_0^d z^{d-3}(d-z)(x-z)^{n-d+1}\dif z.
\end{eqnarray*}
We denote $V=L^n=(-K_X)^n$. So we get the estimate $\vol(L-xE)\ge V-\phi(x)$ for any $x\ge 0$ where
\begin{equation}\label{eq-phix}
   \phi(x)=
   \left\{\begin{array}{ll}
     \ell^{-(n-d+1)} x^{n-1} (dn-(d-2)x) &  \text{if }\; 0\le x\le d\\
    \frac{\ell^{-(n-d+1)} n!}{(d-3)!(n-d+1)!}\int_0^d z^{d-3}(d-z)(x-z)^{n-d+1}\, \dif z & \text{if }\; x\ge d.
\end{array}
\right.
\end{equation}

Note that when $2\le d\le n$, the function $\phi=\phi(x)$ belongs to the class $\mathscr{F}$ considered in Lemma \ref{lem-cal} (see Example \ref{ex-dn1} for the $d=n+1$ case).  In other words, it is a continuous, piecewise smooth, strictly increasing function of $x\in [0, \infty)$ satisfying $\lim_{x\rightarrow 0}\phi(x)=0$ and $\lim_{x\rightarrow+\infty}\phi(x)=+\infty$. Set
\begin{eqnarray}\label{eq-Phix}
\Phi(x)=\int_0^x \phi(t)dt=
\left\{
\begin{array}{ll}
    \frac{\ell^{-(n-d+1)}x^n}{n+1}(d(n+1)-(d-2)x) & 0\le x\le d \\
    \frac{\ell^{-(n-d+1)}n!}{(d-3)!(n-d+2)!}\int_0^\dif z^{d-3}(d-z)(x-z)^{n-d+2}\dif z & x\ge d. 
\end{array}
\right.
\end{eqnarray}

By using the binomial expansion:
\begin{equation*}
    (x-z)^{n-d+1}=(x-d+d-z)^{n-d+1}=\sum_{j=0}^{n-d+1}\binom{n-d+1}{j}(d-z)^j(x-d)^{n-d+1-j}. 
\end{equation*}
it is straightforward to get the following combinatorial expressions when $x\ge d$:
    \begin{eqnarray}
    \phi(x)
    &=&\ell^{-(n-d+1)}\sum_{j=0}^{n-d+1}\binom{n}{j}\cdot (n-d+2-j) d^{n-j}(x-d)^j,\label{eq-phibi}\\
\Phi(x)
&=& \ell^{-(n-d+1)}\sum_{j=0}^{n-d+2} \binom{n+1}{j}\frac{n-d+3-j}{n+1}d^{n+1-j}(x-d)^j\label{eq-Phibi}
\end{eqnarray}
    Set $T=\sup\{x: V-\phi(x)=0\}$. If $T\le d$, then we get $V=\phi(T)\le \phi(d)=\ell^{-(n-d+1)}d^{n}(n-d+2)\le d^n (n-d+2)$. Note that the number $d^n(n-d+2)$ is nothing but the volume of the Fano hypersurface in $\IP^{n+1}$ of degree $n-d+2$ which is always less than $2n^n$ if $2<d<n$. 

\begin{rmk}\label{rmk-2nn}
    Setting $d=2$ in the formula \eqref{eq-phibi}-\eqref{eq-Phibi}, we also recover functions from Example \ref{ex-2nn} (the $\ell^{1-n}$ factor comes the fact that here we are using the $\ell^{\oplus (n-1)}$ blowup): 
    \begin{equation*}
        \phi(x)=2n\cdot \ell^{1-n}x^{n-1}, \quad \Phi(x)= 2 \ell^{1-n}x^n.
    \end{equation*}
\end{rmk}

So we can assume $T\ge d$.
Now we set $\ell=2$ and $A=A_X(v_2)=(d-2)+2(n-d+1)=2n-d$. The equation $\Phi(T)=(T-A)\phi(T)$ from \eqref{eq-pseudoT} becomes:
\begin{equation}\label{eq-key}
    \frac{1}{n-d+2}\int_0^d z^{d-3}(d-z)(T-z)^{n-d+2}\, \dif z=
    (T-(2n-d))\int_0^d z^{d-3}(d-z)(T-z)^{n-d+1}\, \dif z.
\end{equation}
When $2\le d\le n$, by Corollary \ref{cor-calbd}, there exists a unique solution $T=T(d, n)$ to the above equation.

\begin{ex}\label{ex-dn1}
When $d=n+1$, the expression \eqref{eq-phix} becomes:
$\phi(x)=(n+1)n x^{n-1}-(n-1)x^n$ if $x\le n+1$ and $\phi(x)=(n+1)^n$ if $x\ge n+1$. 
We then get the estimate: $\vol(-K_X-xE)\ge \vol(X)-(n+1)^n$ for any $x\ge n+1$, which implies the (un-interesting) volume estimate $\vol(X)\le (n+1)^n=\vol(\IP^n)$.
\end{ex}
\begin{ex}\label{ex-d=n}
When $d=n$, the expressions \eqref{eq-phibi}-\eqref{eq-Phibi} become: 
\begin{eqnarray}
    \phi(x)&=&\ell^{-1}\left(n \cdot n^{n-1}\cdot (x-n)+2 n^n\right)
    =\ell^{-1}n^n(x+2-n),\label{eq-phidn} \\
    \Phi(x)
    &=&\frac{\ell^{-1}n^n}{2}\left(x^2-2(n-2)x+\frac{n}{(n+1)}(n-1)(n-2)\right).\nonumber 
\end{eqnarray}
 When $d=n$ and for the case of $(1^{\oplus(d-2)},\ell^{\oplus(n-d+1)})$-weighted blowup, equation $(T-A(v_\ell))\phi(T)=\Phi(T)$ is equivalent to $T^2-2(n-2+\ell) T+\left(\frac{n(n-2)(n+3)}{2(n+1)}+(2\ell-4)\right)=0$. One can solve \begin{equation*}
     T=\left\{
     \begin{array}{ll}
     (n-1)+\sqrt{\frac{3(n-1)}{n+1}} & \text{for } \ell=1; \\
     n+ \sqrt{\frac{6n}{n+1}}& \text{for } \ell=2.
     \end{array}
     \right.
 \end{equation*}
 So we get the volume upper bound $\vol(X)\le \phi_\ell(T)$ which is given by:
 \begin{equation*}
 \phi_\ell(T)=
 \left\{
     \begin{array}{ll}
     n^n (T-n+2)=n^n (1+\sqrt{\frac{3(n-1)}{n+1}}) & \text{for } \ell=1\\
     \frac{1}{2}n^n(T-n+2)=n^n\left(1+\sqrt{\frac{3n}{2(n+1)}}\right) & \text{for } \ell=2.
     \end{array}
     \right.
 \end{equation*}
 Note that when $n=2$, $\phi_1(T)=\phi_2(T)=\vol(Q)=8$. When $n\ge 3$, $\phi_1(T)>\phi_2(T)>2n^n=\vol(Q)$. To get the sharp upper bound of $\vol(X)=(-K_X)^n$ for all $d=n\ge 3$, we are going to use the classification results stated in Theorem \ref{thm-lX}. 

\end{ex}

By Corollary \ref{cor-calbd},  Proposition \ref{prop-3ton} is equivalent to the following proposition. 
\begin{prop}\label{prop-3&n-1}
    Assume $n=\dim X\ge 4$. 
    Let $3\le d\le n-1$ and $T$ be the solution to equation \eqref{eq-key}. We have $\phi(T)<2n^n$ where $\phi$ was defined in \eqref{eq-phix}. As a consequence, we know that if $X$ is a K-semistable Fano manifold $3\le l_X=d\le n-1$, then $\vol(X)<2n^n$. 
\end{prop}

\begin{rmk}
Example \ref{ex-d=n} shows that for $d=n$ the estimate $\phi(T)<2n^n$ is not true. Surprisingly, Proposition \ref{prop-3&n-1} implies that the $(1^{\oplus (d-2)}, 2^{\oplus (n-d+1)})$ weighted blow-up is sufficient to establish the sharp estimate for all $3\le d\le n-1$.
Even though the statement appears elementary, due to the complicated expansion from \eqref{eq-phibi}-\eqref{eq-Phibi} it seems not so easy to find a simple proof for all pairs of positive integers $(d,n)$ satisfying $3\le d\le n-1$. Motivated by some numerical calculations, we managed to get the desired estimate using some tools from calculus and probability theory. See Figure \ref{fig-n7} for the numerical illustration for $n=7$ and $2\le d\le 7$. 
\end{rmk}

The rest of this subsection is devoted to explain the proof of Proposition \ref{prop-3&n-1}. We first consider the case $d=n-1$. Note that equation $(T-A_X(E))\phi(T)=\Phi(T)$ appeared in Corollary \ref{cor-calbd} is a degree $n-d+2$ polynomial equation. When $d=n-1$, the corresponding equation is a cubic equation in $T$ for which we can explicitly compute the root $T$ using Cardano's formula.
\begin{prop}\label{prop-n-1}
Proposition \ref{prop-3&n-1} holds when $d=n-1\ge 3$.
\end{prop}
\begin{proof}
        When $d=n-1$ and $\ell=2$, the expressions from \eqref{eq-phibi}-\eqref{eq-Phibi} become: 
    \begin{eqnarray*}
        \phi(x) &=& \frac{1}{2^2}\left(\frac{n(n-1)^{n-1}}{2}(x-n+1)^2+2n(n-1)^{n-1}(x-n+1)+3(n-1)^n\right),\\
        \Phi(x) &=& \frac{1}{2^2}(\frac{n(n-1)^{n-1}}{6}(x-n+1)^3+n(n-1)^{n-1}(x-n+1)^2\\
        &&\, +3(n-1)^n(x-n+1)+\frac{4}{n+1}(n-1)^{n+1}).
    \end{eqnarray*}
    The equation $(T-A(E))\phi(T)=\Phi(T)$ is equivalent to $\tau^3-12\tau-\frac{6(n-1)(5n+1)}{n(n+1)}=0$ where $\tau=T-n+1>2n-d-n+1=2$.  By Cardano's formula,
    \begin{eqnarray*}
        \tau = \sqrt[3]{-\frac{q}{2}+\sqrt{\frac{q^2}{4}-64}}+\sqrt[3]{-\frac{q}{2}-\sqrt{\frac{q^2}{4}-64}},
    \end{eqnarray*}
    where $q=-\frac{6(n-1)(5n+1)}{n(n+1)}$. Hence, $\vol(X)\leq\phi(T)=\frac{n(n-1)^{n-1}}{8}((\tau+2)^2+2-\frac{6}{n})\eqqcolon F(n)$. We claim that for $n\geq 4$, we have $F(n)< 2n^n=\vol(Q)$.  
    
    The root $\tau$ satisfies $\tau^3-12\tau=C(n)$, where $C(n)=\frac{6(n-1)(5n+1)}{n(n+1)}$. Note that $f(\tau)=\tau^3-12\tau$ is monotone increasing when $\tau \geq 2$.  And $C(n)$ is monotone increasing with respect to $n$. So $\tau = f^{-1}(C(n))$ is monotone increasing. And $\lim_{n\rightarrow +\infty}\tau(n) = \sqrt[3]{15+\sqrt{161}}+\sqrt[3]{15-\sqrt{161}}\eqqcolon\tau_{\infty}$. So, in particular, we have $\tau \leq \tau_{\infty}$. We further consider \begin{eqnarray*}
        \frac{F(n)}{2n^n}&\leq& \frac{n(n-1)^{n-1}}{8\cdot 2n^n}\left((\tau_{\infty}+2)^2+2-\frac{6}{n}\right)\\
        &\leq&\frac{1}{16\left(1+\frac{1}{n-1}\right)^{n-1}}\left(\left(\sqrt[3]{15+\sqrt{161}}+\sqrt[3]{15-\sqrt{161}}+2\right)^2+2\right)\eqqcolon\gamma(n).
    \end{eqnarray*}
    Note that $\gamma(n)$ is monotone decreasing with respect to $n$ and $\lim_{n\rightarrow\infty}\gamma(n)=\frac{1}{16e}(\sqrt[3]{15+\sqrt{161}}+\sqrt[3]{15-\sqrt{161}}+2)^2<1$. So $\gamma(n)<1$ when $n$ is sufficiently large. By numerical method, we can show that $\gamma(n)<1$ when $n\geq 19$. By direct computation, one can check that $F(n)<2n^n$ when $4\le n\le 18$ (one could also compute that when $n=3$ and $d=2$ so that we are in the case of trivial normal bundle, $F(3)=2\cdot 3^3=54$). 
\end{proof}
\begin{rmk}\label{rmk-12}
     If we use the valuation $v_1$ in the case of $d=n-1$, we can do a similar computation as in the proof of Proposition \ref{prop-n-1}, then
     $$
     \tau=T-n+1=-1+\sqrt[3]{-1+\frac{q}{2}+\sqrt{\frac{q^2}{4}-q}}+\sqrt[3]{-1+\frac{q}{2}-\sqrt{\frac{q^2}{4}-q}}
     $$
     where $q=\frac{12(n-1)^2}{n(n+1)}$. In this case, $\phi(T)=\frac{n(n-1)^{n-1}}{2}\tau^2+2n(n-1)^{n-1}\tau+3(n-1)^n>2n^n=\vol(Q)$ when $n\geq 4$. This shows that the use of $(1^{\oplus(d-2)},2^{\oplus(n-d+1)})$-weighted blowup is crucial for the proof of Proposition \ref{prop-3ton} and hence for Theorem \ref{thm-main}. 
\end{rmk}

    \begin{proof}[Completion of proof of Proposition \ref{prop-3&n-1}]
Because of Proposition \ref{prop-n-1}, we just need to prove Proposition \ref{prop-3&n-1} when $3\le d\le n-2$. For simplicity, set
\begin{eqnarray}\label{eq-Psi}
    \Psi(x)&=&(x-(2n-d))\phi(x)-\Phi(x). 
\end{eqnarray}
Note that $\Psi'(x)=(x-(2n-d))\phi'(x)$ since $\phi(x)=\Phi'(x)$. Since $\phi'(x)>0$ for $x>0$, we have the two useful facts:
\begin{itemize}
     \item $\Psi(x)$ is negative for $x\in (0, T)$ and positive for $x\in (T, +\infty)$, where $T$ is the unique solution to $\Psi(T)=0$ (see Corollary \ref{cor-calbd}).  
    \item $x>2n-d+1$, $\Psi'(x)>\phi'(x)$.
\end{itemize} 
Now we claim that $x_1=2n-d+2$ satisfies the following two inequalities
   \begin{equation}\label{eq-goal0} 
   \Psi(x_1)<0 \quad \text{ and } \quad 
   \phi(x_1)-2n^n-\Psi(x_1)<0.
   \end{equation} 
Assuming this, we can complete the proof of the proposition. To see this, we note that by the first fact above, $\Psi(x_1)<0$ implies that $x_1<T$. Then we can estimate:
   \begin{eqnarray*}
    \phi(T)&=&\phi(x_1)+\int_{x_1}^T \phi'(x)\dif x<\phi(x_1)+\int_{x_1}^T\Psi'(x)\dif x\\
      &=&\phi(x_1)+\Psi(T)-\Psi(x_1)<2n^n. 
   \end{eqnarray*}
Note that we used the 2nd fact pointed out above since $x_1>2n-d+1$, and the vanishing $\Psi(T)=0$. 
Finally we need to verify the above claim that the two estimates in \eqref{eq-goal0} are true. 
In other words, we need to show that for any $3\le d\le n-2$, the following two inequalities are true:
\begin{equation}\label{eq-x1T}
    \Psi(2n-d+2)=2\phi(2n-d+2)-\Phi(2n-d+2)<0, 
\end{equation}
\begin{equation}\label{eq-goal}
   \phi(2n-d+2)-\Psi(2n-d+2)=\Phi(2n-d+2)-\phi(2n-d+2)<2n^n. 
\end{equation}
The inequality \eqref{eq-x1T} can be quickly verified below in Lemma \ref{lem-x1T}. Regarding the inequality \eqref{eq-goal}, for concrete $n$ and $d$, it can be verified by straight forward calculation. For arbitrary $n$ and $d$, the proof of \eqref{eq-goal} is complicated and is contained in the \textbf{Appendix }\ref{proof-goal}. 
Roughly speaking, we will establish some effective estimates to reduce the proof of \eqref{eq-goal} to finitely many cases of $(n,d)$ for which we can verify numerically.
\end{proof}
\begin{lem}\label{lem-x1T}
For $3\le d\le n-2$, 
    $\Phi(2n-d+2)>2\phi(2n-d+2)$. 
\end{lem}
\begin{proof}
    Set $x_1=2n-d+2$ which is bigger than $d$ and
    $C=\frac{(d-3)!}{2^{n-d-1}(n-d+2)!}>0$.
We will use the integral formula in \eqref{eq-phix}-\eqref{eq-Phix} to calculate: 
\begin{eqnarray*}
  &&\Phi(x_1)-2\phi(x_1)=C\int_0^d z^{d-3}(d-z)(x_1-z)^{n-d+1}\left[x_1-z-2(n-d+2)\right]\dif z\\
  &=&C\int_0^d z^{d-3}(d-z)\left[2(n-d+1)+d-z\right]^{n-d+1}(d-z-2)\dif z\\
  &=&C \int_0^d z^{d-3}(d-z)\sum_{j=0}^{n-d+1}\binom{n-d+1}{j}2^j(n-d+1)^j(d-z)^{n-d+1-j}((d-z)-2)dz\\
  &=&C\sum_{j=0}^{n-d+1}\binom{n-d+1}{j}d^{n-j}\cdot 2^j(n-d+1)^j\\
  &&\quad\cdot\left(d\cdot B(d-2, n-d-j+4)-2 B(d-2, n-d+3-j)\right)\\
  &=&\tilde{C}\sum_{j=0}^{n-d+1}2^j(n-d+1)^j d^{n-j}\frac{n-d+2-j}{j!(n-j)!} \cdot \left(\frac{d(n-d+3-j)}{n+1-j}-2\right)>0,
\end{eqnarray*}
where $\tilde{C}=(n-d+1)!\cdot C$,  $\frac{d(n-d+3-j)}{n+1-j}-2=\frac{(d-2)(n-d+1-j)}{n+1-j}\geq 0$ with equality only if $j=n-d+1$ and $B(\alpha,\beta)=\int_0^1 x^{\alpha-1}(1-x)^{\beta-1}\dif x$ denotes the Beta function for $\alpha,\beta \in \mathbb{R}_{>0}$.
\end{proof}

\subsection{Completion of the proof of Theorem \ref{thm-main}}
\begin{proof}[Proof of Theorem \ref{thm-main}]
We use the notation $l_X$ from \eqref{eq-lX1}-\eqref{eq-lX}.
If $l_X=n+1$, then by Theorem \ref{thm-lX}.(1), $X$ must be $\IP^n$. If $l_X=n$ and $\rho(X)=1$ then by Proposition \ref{prop-lxn}, we have $(-K_X)^n<2n^n$ unless $X$ is isomorphic to $Q^n$. If $l_X=n$ and $\rho(X)\ge 2$, then by Theorem \ref{thm-lX}.(2), we can conclude that $X$ must be the blowup of $\IP^n$ along a smooth codimension two subvariety $Y$ of degree $d_Y\in \{1, \dots, n\}$ in a hyperplane. When $d_Y\geq 2$, by Lemma \ref{lem:CDvolume} below we know $(-K_{\bl_Y \IP^n})^n<2n^n$. When $d_Y=1$, by Lemma \ref{lem-CDunstable} below we know that $\bl_{\IP^{n-2}}\IP^n$ is not a K-semistable Fano manifold, so it is excluded. Thus, we verified the Theorem \ref{thm-main} in the case $l_X\ge n$. When $l_X<n$, there are minimal rational curves of degree $d\in \{2,\dots, n-1\}$. For $3\le d\le n-1$, 
    by Proposition \ref{prop-3&n-1}, $\vol(X)<2n^n$. When $d=2$, the minimal rational curve has a trivial normal bundle. If its image has singular points, then by Proposition \ref{prop-deg2} $\vol(X)<2n^n$ when $n\ge 3$. For $n=2$, the sharp volume bound of course follows from the classification of del Pezzo surfaces. Note that in this case, the quadric surface coincides with $\mathbb{P}^1\times \mathbb{P}^1$. 
    Finally, if the minimal rational curve with trivial normal bundle is an embedding, then by Theorem \ref{thm-ntrivial} we get $\vol(X)\le 2n^n$ with equality if and only if $X\cong \IP^1\times \IP^{n-1}$. 
\end{proof}
\begin{lem}\label{lem:CDvolume}
    Suppose $X$ is the blowup of $\IP^n$ along a subvariety $Y$ of degree $d_Y\in \{1,\cdots,n\}$ that is contained in a hyperplane (See Theorem \ref{thm-lX}). Then
    \begin{equation*}
        (-K_{\bl_Y \IP^n})^n=\left\{\begin{array}{ll}
        2n^n & \text{if } d_Y=1\\
        \frac{d_Y n^n-(n+1-d_Y)^n}{d_Y-1}=:V_{d_Y} & \text{if } 2\le d_Y\le n.
        \end{array}
        \right.
    \end{equation*}
\end{lem}
\begin{proof}
    Denote $d=d_Y$. The normal bundle of $Y$ is given by $N_{Y/\IP^n}=\CO_Y(1)\oplus \CO_Y(d)$. Let $H=\pi^*\CO_{\IP^n}(1)$ be the pullback of a hyperplane section of $\IP^n$. Then,
\begin{eqnarray*}
    (-K_{\bl_Y\IP^n})^n &=& ((n+1)H-E)^n = (n+1)^n+\sum_{k=2}^n \binom{n}{k}(n+1)^{n-k}\cdot(-1)^k(H^{n-k}\cdot E^k)\\
    &=&  (n+1)^n-\sum_{k=2}^n \binom{n}{k}(n+1)^{n-k}\cdot(s_{k-2}(N_{Y/\IP^n})\cdot H^{n-k}|_A)\\
    &=& (n+1)^n-\sum_{k=2}^n \binom{n}{k}(n+1)^{n-k}\cdot (-1)^{k-2}\left(\sum_{i=0}^{k-2} d^i\right)\cdot d,
\end{eqnarray*}
where $s_{k-2}(N_{Y/\IP^n})$ denote the $(k-2)$-th Segre class of normal bundle $N_{Y/\IP^n}$. If $d=1$, then $\sum_{i=0}^{k-2}d^i=k-1$. By the combinatorial identity $\binom{n}{k}(k-1)=n\binom{n-1}{k-1}-\binom{n}{k}$, 
\begin{eqnarray*}
    (-K_{\bl_{\IP^{n-2}}\IP^n})^n
    &=& (n+1)^n-\sum_{k=2}^n (-1)^k\left[n\binom{n-1}{k-1}-\binom{n}{k}\right](n+1)^{n-k} = 2n^n.
\end{eqnarray*}
If $d\gee 2$, then $\sum_{i=0}^{k-2}d^i=(d^{k-1}-1)/(d-1)$. Then,
\begin{eqnarray*}
    (-K_{\bl_{Y}\IP^n})^n
    &=& (n+1)^n-\sum_{k=2}^n (-1)^k\binom{n}{k}\cdot\left(\frac{d^k-d}{d-1}\right)(n+1)^{n-k}\\
    &=& \frac{1}{d-1}\left(dn^n-(n+1-d)^n\right).
\end{eqnarray*}
Note that by L'Hospital's rule $\lim_{d_Y\rightarrow 1}V_{d_Y}=2n^n$ and it is easy to check that $V_{d_Y}<2 n^n$ for $d_Y\ge 2$. 
\end{proof}
\begin{lem}\label{lem-CDunstable}
    Suppose $\pi\colon X\rightarrow \IP^n$ is the blowup of $\IP^n$ along a subvariety $Y\cong \IP^{n-2}$ of degree $d_Y=1$ contained in a hyperplane of $\IP^n$, then $X$ is K-unstable.
\end{lem}
\begin{proof}
    Let $E = \pi^{-1}(\IP^{n-2})$ be the exceptional divisor. And let $H=\pi^*\CO_{\IP^n}(1)$ be the pullback of a hyperplane section of $\IP^n$. Note that the nef cone is $\text{Nef}(X)=\IR_{\gee 0}[H]+\IR_{\gee 0}[H-E]$. Then the divisor $-K_{X}-tE=(n+1)H-(1+t)E$ is nef if and only if $0\lee t\lee n$. When $0\lee t\lee n$, by the similar computation as in Lemma \ref{lem:CDvolume},
    \begin{eqnarray*}
        \vol_{X}(-K_X-tE)&=&((n+1)H-(1+t)E)^n=(n+1)^n-\sum_{k=2}^n \binom{n}{k}(n+1)^{n-k}(-1-t)^{k}(k-1)\\
        &=& (n-t)^{n-1}(2n+(n-1)t).
    \end{eqnarray*}
    We see the pseudo-effective threshold $T_X(E)$ equals to the nef threshold $n$. Then, 
    \begin{eqnarray*}
        S(-K_X;E)=\frac{1}{(-K_X)^n}\int_0^n\vol_X(-K_X-tE) \dif t = 1+\frac{n-1}{2(n+1)}.
    \end{eqnarray*}
    So $A_X(E)-S(-K_X;E)=1-(1+\frac{n-1}{2(n+1)})=-\frac{n-1}{2(n+1)}<0$. By the valuative criterion (Theorem \ref{thm:beta}), we conclude that $X$ is K-unstable and destabilized by the exceptional divisor $E$.
\end{proof}
\begin{rmk}
    Note that $X=\bl_{\IP^{n-2}}\IP^n$ is a toric Fano variety, so one can also prove lemma \ref{lem-CDunstable} via toric geometry. The minimal generators of the fan $\Sigma\subseteq \IR^n$ is given by $e_1=(1,0\cdots,0),\cdots, e_n=(0,\cdots,0,1), e_0=(-1,-1,\cdots,-1), f=(1,1,0,\cdots,0)$. Here the ray generated by $f$ corresponds to the exceptional divisor $E=\pi^{-1}(\IP^{n-2})$ by the cone-orbit correspondence. Then the moment polytope with respect to $-K_X$ is given by
    \begin{eqnarray*}
        P_X =\{(x_1,\cdots,x_n)\in \IR^n \mid x_1+x_2\geq -1,\,\sum_{i=1}^n x_i\leq 1,\,x_j\geq -1 \,\,\text{for}\,\, j=1,\cdots,n\}.
    \end{eqnarray*}
    By calculus, one can compute the barycenter $\bc(P)=(\frac{n-1}{4(n+1)},\frac{n-1}{4(n+1)}, -\frac{1}{2(n+1)},\cdots, -\frac{1}{2(n+1)})\in \IR^n$, which does not coincide with the origin. Thus, $X$ is K-unstable, so this gives another proof of Lemma \ref{lem-CDunstable}. In fact, by the formula $\delta(X,-K_X)=\min_{\rho\in \Sigma(1)}\frac{1}{\langle \bc(P),u_{\rho}\rangle+1}$ in \cite[Section 7]{BJ20}, we know $\delta(X) = \frac{2n+2}{3n+1}<1$ and it is minimized at $\rho=f=(1,1,0\cdots,0)$. So the exceptional divisor $E$ in Lemma \ref{lem-CDunstable} is indeed the $\delta$-minimizer of $X$.
\end{rmk}

\begin{rmk}
    When $X$ is the blowup of $\IP^n$ along a smooth codimension two subvariety $Y$ of degree $d_Y\geq 2$ in a hyperplane $H$, K. Fujita showed in \cite[Lemma 9.7]{Fuj16} that $X$ is K-unstable and destabilized by the strict transform of $H$ on $X$. Although we do not need this result in our proof of Theorem \ref{thm-main}, we thank the anonymous referee for bringing this fact to our attention.
\end{rmk}
As an immediate application of Theorem \ref{thm-main}, we prove the ODP conjecture (Conjecture \ref{conj:ODP}) for Fano cones over K-semistable Fano manifolds.  
\begin{thm}\label{cor:ODPregular}
    Suppose $L = r^{-1}(-K_X)$ is an ample line bundle over an $n$-dimensional smooth K-semistable  Fano manifold $X$ for some $r\in \IN^*$, then ODP conjecture (Conjecture \ref{conj:ODP}) holds for the affine cone $\CC(X,L)\coloneqq \spec(\oplus_{k=0}^{+\infty}H^0(X,kL))$ with the cone vertex $o$.
\end{thm}
\begin{proof}
    We denote $\pi\colon\bl_o \CC \rightarrow \CC$ be the blowup of $\CC$ along the cone vertex $o$. By the result of \cite{Li17,LL19}, we know $X$ is K-semistable if and only if the minimizer of normalized volume of $(\CC,o)$ is attained by the canonical valuation $v=\ord_E$, which is the divisorial valuation associated to the exceptional divisor $E=\pi^{-1}(o)\cong X$. From \cite[Section 3.1]{Kol13}, we know 
    \begin{eqnarray*}
        K_{\bl_o \CC}+(1-r)E\sim_{\IQ}0\sim_{\IQ}\pi^*(K_{\CC}).
    \end{eqnarray*}
    Then $A_{\CC}(\ord_E)=1+\ord_E(K_{\bl_o \CC}-\pi^{*}K_{\CC})=1+(r-1)=r$ and  $\vol(\ord_E)=L^n$. Then 
    $$
    \hvol(\CC,o)=A_{\CC}(\ord_E)^{n+1}\cdot \vol(v)=r^{n+1}\cdot L^n=r(-K_X)^n.
    $$
    Let $i(X)\coloneqq \max\{r\in \IN^*\mid -K_X\sim_{\IQ}rL\,\text{for some ample line bundle }\, L\}$ be the Fano index of $X$. Then the above argument shows that $\hvol(\CC,o)\lee i(X)\cdot (-K_X)^n$. We know that for smooth Fano manifold, the Fano index always satisfies $i(X)\lee n+1$ and $i(X)=n+1$ if and only if $X\cong \IP^n$ and $i(X)=n$ if and only if $X\cong Q\subset \IP^{n+1}$ is a smooth quadric hypersurface (\cite{KO73}). 
    
    If $X\cong \IP^n$ and $r=i(\IP^n)=n+1$, then $(\CC(\IP^n,\CO(1)),o)\cong (\IC^{n+1},0)$ is a smooth point, which is excluded in the ODP conjecture. If $r<i(\IP^n)=n+1$, then there exists some $p\in \IZ_{>1}$ such that $rp=i(\IP^n)=n+1$. Then $\hvol(\CC(\IP^n,L),o)=r(-K_{\IP^n})^n=(n+1)^{n+1}/p<2n^{n+1}$. So we can assume $X\ncong \IP^n$. Then by Theorem \ref{thm-main}, we have $(-K_X)^{n}\lee 2n^n$ and the equality holds if $X\cong \IP^1\times \IP^{n-1}$ or $X$ is a smooth quadric hypersurface $Q$. Combined with Kobayashi-Ochiai's result $i(X)\lee n$, we get 
    \begin{eqnarray*}
        \hvol(\CC,o)\lee i(X)\cdot(-K_X)^n\lee 2n^{n+1}.
    \end{eqnarray*}
    The equality holds if and only if $i(X)=n$ and $(-K_X)^n=2n^n$, which implies $(\CC,o)\cong (\CC(Q,\CO_Q(1)),o)$ is the ODP singularity (the case when $X\cong \IP^1\times \IP^{n-1}$ is excluded since the Fano index $i(\IP^1\times \IP^{n-1})\leq 2$).
\end{proof}

\subsection{Examples and questions}

\begin{ex}\label{ex-quadric}
    Our consideration of weighted blowup is motivated by a special example of a weighted blowup over the quadric hypersurface: 
    $X=Q^n=\{Z_0Z_1+Z_2Z_3+Z_4^2+\cdots+Z_{n+1}^2=0\}\subset \IP^{n+1}$ for $n\ge 3$. 
    Fix a minimal rational curve which is the line $$C\cong \IP^1=\{Z_1=Z_3=Z_4=\cdots=0\}=\{[t:0:s:0:0:\cdots: 0]\colon [t,s]\in \IP^1\}\subset \IP^{n+1}.$$ 
    On the affine chart $U_0=\{Z_0\neq 0\}$, set
    $u_i=z_i /z_0$ with $i\neq 0$. Then the equations of $X$ and $C$ become:
    \begin{eqnarray*}
        X: && u_1+u_2u_3+u_4^2+\cdots+u_{n+1}^2=0\Longrightarrow u_1=-u_2u_3-u_4^2-\cdots-u_{n+2}^2,\\
        C: && u_3=u_4=\cdots=u_{n+1}=0.
    \end{eqnarray*}
   In particular, we choose coordinates $\mathbf{u}=\{u_2, u_3, u_4, \cdots, u_{n+1}\}$ as local coordinates on $X$.   
    Similarly, on the affine chart $U_1=\{Z_2\neq 0\}$, if we set $v_j=z_j/z_2$ with $j\neq 2$. The local coordinate is given by $\Bv=(v_0,v_1,v_4\cdots,v_{n+1})$. The equations of $X$, $C$ 
    are given by:
    \begin{eqnarray*}
        X: && v_0v_1+v_3+v_4^2+\cdots+ v_{n+1}^2=0\Longrightarrow v_3=-v_0v_1-v_4^2-\cdots-v_{n+1}^2,\\
        C: && v_1=v_4=\cdots=v_{n+1}=0. 
    \end{eqnarray*}
    \begin{equation*}
        v_0=\frac{Z_2}{Z_2}=u_2^{-1}, 
        v_1=\frac{u_1}{u_2}=-u_3-\sum_{k=4}^{n+1}\frac{u_k^2}{u_2}, v_k=\frac{Z_k}{Z_2}=\frac{Z_k}{Z_0}\frac{Z_0}{Z_2}=u_k\cdot u_2^{-1}
    \end{equation*}

    The conormal bundle of $C\setminus \{x\}$ is generated by $$\{d v_1=-du_3, d v_4=u_2^{-1} du_4, \cdots, d v_{n+1}=u_2^{-1} du_{n+1}\}$$
    which shows in particular that $N^{\vee}_{C/X}=\CO\oplus\CO(-1)^{\oplus n-2}$. 
   
    Consider the $\IC^*$ action on $\IP^{n+1}$ with weights $(0, 2, 0, 2, \underbrace{1,\dots, 1}_{n-3})$ on the homogeneous coordinates. 
   Note that the fixed point set of this $\IC^*$ always contains two lines: $C$ and $C_1=\{[0:t:0:s:0:\cdots: 0]; [t,s]\in \IP^1\}$. When $n\ge 4$, the fixed point set has another component $Q^{n-4}=\{Z_0=Z_1=Z_2=Z_3=0\}\cap Q^n\hookrightarrow Q^n$. Note that the $\IC^*$ action preserves the quadric hypersurface and gives rise to a divisorial valuation whose center on $X$ is the line $C$. This divisorial valuation can also be obtained via a weighted blowup along the line $C$. Precisely, in the 
    $\mathbf{u}$-coordinates, we perform the weighted blow along $C\cap U_0=\{u_3=u_4=\cdots=u_{n+1}=0\}$
    with weights $(0, 2, 1, \dots, 1)$. In the $\mathbf{v}$ coordinates, we also have the weights $(0, 2, 1,\dots, 1)$. We see that the weighted blowup $\mu: \hat{X}\rightarrow X$ is indeed globally defined and the exceptional divisor $E$ is isomorphis to a weighted projective $\IP(2, \underbrace{1, \dots, 1}_{n-2})$-bundle over $\IP^1$.
    
    We can calculate the volume function $\vol(-K_X-x E)=n^n \vol(H-\xi E)$ where $\xi=\frac{x}{n}$. 
    Corresponding to the geometry of fixed point sets mentioned earlier, the situation turns out to be different for $n=3$ and $n\ge 4$. 
    When $n=3$, the nef threshold (Seshadri constant) of $H$ with respect to $E$ coincides with the pseudo-effective threshold and is equal to 2. We have the expression:
    \begin{eqnarray*}
        \vol(-K_X-xE)=27\cdot \vol(H-\xi E)=27\cdot\left(2-\frac{3}{2}\xi+\frac{1}{2}\xi^2\right), \quad 0\le \xi\le 2.
    \end{eqnarray*}
    When $n\ge 4$, the nef threshold of $H$ with respect to $E$ is 1 while the pseudoeffective threshold of $H$ with respect to $E$ is 2. The volume function consists of two smooth pieces:
    \begin{eqnarray}\label{eq-Qvol}
    \vol(-K_X-x E)&=&n^n\vol(H-\xi E)\\
    &=& n^n\cdot 
    \left\{
    \begin{array}{ll}
     2-\frac{n}{2} \xi^{n-1}+\frac{n-2}{2}\xi^n, & \xi=\frac{x}{n}\in [0, 1]\nonumber \\
    \frac{n}{2}(2-\xi)^{n-1}-\frac{n-2}{2}(2-\xi)^n, &  \xi=\frac{x}{n}\in [1,2]. 
    \end{array}
    \right.
\end{eqnarray}
In fact, this expression holds for any $n\ge 3$, since when $n=3$ the second piece and the first piece connect smoothly and become one whole smooth piece. 
To get the above formula, we can first calculate the volume function before the nef threshold by calculating the intersection:
\begin{eqnarray*}
    (\mu^*H-\xi E)^n&=&\sum_{k=0}^{n}\binom{n}{k}\mu^*H^{n-k}\cdot \xi^k (-E)^{k}\\
    &=&H^n+n\xi^{n-1}\pi^*H\cdot \xi^{n-1} (-E)^{n-1}+\xi^n (-E)^n\\
    &=&2-\frac{n}{2}\xi^{n-1}+\xi^n\frac{n-2}{2}
    \end{eqnarray*}
    where we used a weighted analogue of the usual intersection formula via Segre classes for projective bundles. It is not surprising that the expression of this piece coincides with the piece in the expression of \eqref{eq-phix} when $x=n\xi \le n$.

    The second piece is more difficult to get. One way to get it is to use the fact that the measure $-\frac{1}{n!} d \vol(H-tE)$ is the Duistermaat-Heckman measure of the previously mentioned Hamiltonian $S^1$ action (see \cite{BHJ17}) and then use the symmetry of its density function around $\xi=1$ in our case. We can also directly calculate the Duistermaat-Heckman measure of $S^1$-action by using a well-established localization/Fourier formula (\cite[7.4]{BGV92}) and for the interest of the reader we will provide such a derivation in Appendix B. It is straightforward to verify the estimate $\vol(-K_X-xE)\ge V-\phi(x)$ (see \eqref{eq-phidn}) for this example which is reduced to the inequality: for $n\le x\le 2n$, 
    \begin{eqnarray*}
        F(x):=\frac{n^2}{2}(2n-x)^{n-1}-\frac{n-2}{2}(2n-x)^n+\frac{1}{2} n^n(x+2-n)- 2n^n\ge 0. 
    \end{eqnarray*}
    This can be easily verified by the following facts: $F(n)=F'(n)=0$ and $F''(x)=\frac{n}{2}(n-1)(n-2)(2n-x)^{n-3}(x-n)>0$ for $n<x<2n$.

\end{ex}

\begin{ex} 
    Let $X\subset \IP^{n+1}$ be a Fano hypersurface of degree $1\le b\le n$. Then the minimal rational curve is of splitting type $\CO(1)^{n-b}\oplus \CO^{b-1}$ so that $d=n-b+2$ (see \cite[Exercise V.4.4]{Kol96}). The volume of $X$ is given by $V_d=d^n(n-d+2)=(n+2-b)^n\cdot b$. One can easily verify that $V_d$ is strictly smaller than $\vol(\IP^{d-1}\times \IP^{n-d+1})$ except when $b=1$ or $b=2$.  
\end{ex}

\begin{ques}
   Assume that a K-semistable Fano manifold $X$ contains a minimal rational curve of degree $3\le d\le n-1$ in $X$. Is it true that $\vol(X)=(-K_X)^n\le \vol(\IP^{d-1}\times \IP^{n-d+1})$? 
\end{ques}

In the first version, this was claimed as a result. However due to the issue discussed in Remark \ref{rmk-Zhuang}, the question is left un-answered. 
 On the other hand, our current results imply that the answer is affirmative if we instead assume $d\in \{2,n,n+1 \}$. 

We end this paper by finding the minimal possible anti-canonical volume of $n$-dimensional (K-semistable) Fano manifolds. It is not a well-posed problem for general $\IQ$-Fano varieties. Even assuming K-semistablity, there exists a sequence of K-semistable $\IQ$-Fano varieties with volume $(-K_X)^n$ tends to zero (see \cite[Example 1.4(2)]{Jia17}), so it is in general not possible to determine an optimal positive lower bound. When we restrict to $n$-dimensional smooth Fano manifold $X$, the volume $(-K_X)^n$ is always a positive integer, so a priori $(-K_X)^n\gee 1$. 
\begin{ex}\label{ex:V1even}
    Suppose $n\gee 2$ is an integer. Assume $a_0\lee a_1\lee\cdots \lee a_{n+1}$ are integers and we consider a general degree $d$ hypersurface $X_d$ of weighted projective space $\IP(a_0,a_1\cdots,a_{n+1})$ that is well-formed and quasi-smooth. Note that $X_d$ is a Fano variety if $d<\sum_{i=0}^{n+1}a_i$. By the adjunction formula
    $$
    (-K_X)^n = \frac{d\cdot(\sum_{i=0}^{n+1} a_i -d)^n}{\prod_{i=0}^{n+1}a_i}. 
    $$
    If $a_0=\cdots=a_{n-1}=1,\, a_n =2, \,a_{n+1}=n+1$ and $d= 2n+2$. Namely, $X_d$ is a general degree $2n+2$ hypersurface in $\IP(1^n,2,n+1)$. Then $(-K_{X_{2n+2}})^n = 1$.
    By \cite[Thm 8.1]{IF00}, we know that $X_{2n+2}$ is smooth when $n$ is even. By 
    \cite[Thm 1.3]{ST24}, 
    we know $\delta(X_{2n+2})\gee (n+1)/2> 1$, so $X_{2n+2}$ is a smooth K-stable Fano manifold with minimal anti-canonical volume $(-K_{X_{2n+2}})^n=1$ for all even $n$. When $n$ is odd, the general hypersurface $X_{2n+2}\subset \IP(1^n,2,n+1)$ is singular. (After communicating with K. Fujita, we find that this type of example already appears in \cite[Theorem 5.22]{KSC04}.)
\end{ex}
\begin{ex}\label{ex:V2}
    Let $Y$ be the double cover of $\IP^n$ ramified along a degree $2n$ smooth hypersurface $D$. Then by Hurwitz's formula, $K_Y=\phi^*(K_{\IP^n}+\frac{1}{2}D)$. So $(-K_Y)^n = \deg(\phi)\cdot (-K_{\IP^n}-\frac{1}{2}D)^n=2$. It was first shown in \cite[Theorem 3.2]{AGP06} that $Y$ admits K\"ahler-Einstein metric, hence K-stable (See also \cite{Der16} and \cite[Corollary 4.9(3)]{AZ22} for a purely algebraic proof of K-stability). Here we illustrate a different proof based on the result of \cite{LZ20,Zhu21}, we know K-stability of $Y$ is equivalent to the K-stability of the log Fano pair $(\IP^n,\frac{1}{2}D)$. Since we assume $D$ is a smooth hypersurface, then the log Fano pair $(\IP^n, (\frac{n+1}{2n}-\varepsilon)D)$ is klt, so it is uniformly K-stable by \cite[Theorem 2.10]{ADL21} for $0<\varepsilon\ll1$. Since $D\sim_{\IQ} \frac{2n}{n+1}(-K_{\IP^n})$, by the interpolation of K-stability (\cite[Prop 2.13]{ADL19}), we conclude that $(\IP^n,\frac{1}{2}D)$ is K-stable. Thus, $Y$ is a smooth K-stable Fano manifold with volume $(-K_Y)^n = 2$ for all dimensions $n\gee 2$. Note that $Y$ contains a degree 2 free-immersed rational curve with $n-1$ nodes (see \cite[\rm IV.2.12.3]{Kol96}).
\end{ex}
We indeed have the following:
\begin{prop}
    The minimum of anticanonical volumes of $n$-dimensional Fano manifolds is equal to 1 if $n$ is even and is equal to 2 if $n$ is odd. The minimum volume is obtained by K-stable Fano manifolds. 
\end{prop}
This follows from Example \ref{ex:V1even}, Example \ref{ex:V2} and following interesting fact pointed out to us by K. Fujita:

\begin{fact}\label{fact-vol}
    (\cite[Solution to Exercise 5.23, Page 217]{KSC04}) Let $X$ be a smooth projective variety of odd dimension $n$. Then the self-intersection number $(-K_X)^n$ is even. 
\end{fact}

\begin{ques}
   Classify all Fano manifolds that obtain the minimal volume.     
\end{ques}


\appendix

\section{Proof of Proposition \ref{prop-lxn}}\label{sec-lXn}
The following proof of Proposition \ref{prop-lxn} is motivated by the study of works \cite{CMS02, DH17, Miy04}.  
ChatGPT 5.6 is used for completing the proof, in particular for deriving the key estimate \eqref{eq-VMRT2}. 

    Let $H$ be the primitive ample generator of $\mathrm{Pic}(X)$ and write $-K_X=rH$. Fix a general point $x\in X$ and let $S=S_x$ be the normalized parameter space of an irreducible component of degree $n$ rational curves through $x$. 
    Let 
    \begin{equation*}
        \pi: U=U_x\rightarrow S, \quad ev: U\rightarrow Y=\mathrm{Im}(ev)\subset X
    \end{equation*}
 be the normalized universal family and evaluation map. The fibration $\pi$ is a $\IP^1$-bundle whose general fiber  
 $C$ is an embedded rational curve $f:\IP^1\rightarrow X$ with
 $TX|_C\cong \mathcal{O}(2)\oplus \mathcal{O}_{\IP^1}(1)^{\oplus (n-2)}\oplus \mathcal{O}$. 
 We have $\dim S=-K_X\cdot C-2=n-2$ and $\dim Y=n-1$ (see  \cite{CMS02} and \cite[Proposition 1.3]{Miy04}. Set $m\coloneqq H\cdot C$. Then we have $n=-K_X\cdot C=rm$. The image of the evaluation map $Y$ is an irreducible divisor. Since $\mathrm{Pic}(X)=1$, we can write $Y\sim aH$ with $a\in \IZ_{>0}$. Note that $\dim U=\dim Y=n-1$ and $ev: U\rightarrow Y$ is generically finite. Put $s\coloneqq \deg(ev)\ge 1$. 

Thanks to the work of Kebekus (\cite[Theorem 3.3]{Keb02}), the evaluation map lifts to the blowup of $X$ at a general $x$ so that we have the following diagram: 
\[
\begin{tikzcd}[column sep=large, row sep=large]
& \widehat{X} \arrow[d, "\text{blow-up of }x"] \\
U_x
    \arrow[r, "ev"]
    \arrow[d, "\pi"']
    \arrow[ur,  "\widehat{ev}"]
&
X \\
S_x &
\end{tikzcd}
\]
There is a section $\sigma\cong S_x$ of the fibration $\pi$ that is contained in $ev^{-1}(x)$ and $\tau_x=\hat{ev}|_{\sigma}$ is the tangent morphism. 
Hwang-Mok \cite{HM04} proved that the tangent morphism $\tau_x: \sigma\rightarrow \mathcal{V}_x\subset \IP(T_xX)$
is birational onto the VMRT (Varieties of Minimal Rational Tangents) $\mathcal{V}_x$. Set
 \begin{equation*}
 \xi:=\tau_x^*\mathcal{O}_{\IP(T_xX)}(1), \quad 
     c:=\xi^{n-2}=\deg \mathcal{V}_x. 
 \end{equation*}
 By \cite[Theorem 3.3]{Keb02}, for a general $x\in X$, rational curves parametrized by $S_x$ are immersion at $x\in X$. We then know that $\sigma|_\sigma=N_{\sigma/U_x}\cong \tau_x^*\mathcal{O}_{\IP(T_xX)}(-1)=\mathcal{O}_\sigma(-\xi)$. 
 
We claim that the following identity and inequality hold:
 \begin{equation}\label{eq-VMRT1}
     sa H^n=m^{n-1}c
 \end{equation}
 \begin{equation}\label{eq-VMRT2}
     c\le am+1.
 \end{equation}
The proposition follows from the claim. Indeed, since $s, a, m$ are positive integers, we combine \eqref{eq-VMRT1}-\eqref{eq-VMRT2} to get:
 \begin{equation*}
     H^n=\frac{m^{n-1}c}{sa}\le \frac{m^{n-1}(am+1)}{sa}=m^n+\frac{m^{n-1}}{a}\le 2m^n. 
 \end{equation*} 
 Then using $rm=n$, we get:
 \begin{equation*}
     (-K_X)^n=r^n H^n\le 2r^n m^n=2n^n. 
 \end{equation*}
 The equality holds only if $s=a=m=1$.  Then we get $n=rm=r$ and hence $-K_X=nH$. By Kobayashi-Ochiai's classification result \cite{KO73} for Fano manifolds with large Fano indices, we must have $X\cong Q^n$. 

So we are left to prove the claim \eqref{eq-VMRT1}-\eqref{eq-VMRT2}.  To prove \eqref{eq-VMRT1}, consider the divisor $B=ev^*H$. Since it has degree $m$ on each fiber, we have
$B\equiv m\sigma+\pi^*M$ for a divisor $M$ on $S_x$ where $\equiv$ means numerical equivalence. Since $ev(\sigma)=x$, $B|_\sigma\equiv 0$. By using the isomorphism of $\sigma$ with $S_x$, we have $M\equiv -m \sigma|_\sigma\equiv m \xi$. So we get $B\equiv m(\sigma+\pi^*\xi)$. Note that $(\sigma+\pi^*\xi)|_{\sigma}\equiv 0$ and $(\pi^*\xi)^{n-1}=0$, we have:
\begin{equation*}
    B^{n-1}=m^{n-1}\pi^*\xi^{n-2}\cdot \sigma=m^{n-1}\deg (\mathcal{V}_x)=m^{n-1}c. 
\end{equation*}
On the other hand, by the projection formula, we have $B^{n-1}=s H^{n-1}\cdot Y=sa H^{n}$. So we get the equality \eqref{eq-VMRT1}. 

To prove \eqref{eq-VMRT2}, we just need to prove $\mult_xY\le Y\cdot C+1=am+1$. This is because that $\mathcal{V}_x$ is an irreducible component of the reduced support of the projectivized tangent-cone of $Y$ at $x$, which implies $\deg(\mathcal{V}_x)\le \mult_x Y$. To estimate $\mult_xY$, we choose a general fiber $C$ which is a minimal rational curve passing through $x\in X$ and a smooth point of $Y$. Then there is an intrinsically defined nonzero map 
\begin{equation*}
    \phi: N_{C/X}\rightarrow \mathcal{O}(Y)|_C. 
\end{equation*}
If we let $\{f=0\}$ be a local equation of $Y$ near $x\in X$ and choose analytic coordinates $\{t, z_1, \dots, z_{n-1}\}$ centered at $x$ such that $C=\{z_1=\cdots=z_{n-1}=0\}$ and $t$ is a parameter on $C\cong \IP^1$, then the above map is represented by the vector $\left(\partial f/\partial z_1, \dots, \partial f/\partial z_{n-1}\right)|_{C}$ under natural local trivialization. If $k=\mult_xY$, then this map factors through a nonzero morphism 
$N_{C/X}\rightarrow \mathcal{O}_C(Y)(-(k-1)x)$. 
Since $C\cong \IP^1$ and $Y\cdot C=am$, the target is $\mathcal{O}_{\IP^1}(am-k+1)$. Since $N_{C/X}\cong \mathcal{O}_{\IP^1}(1)^{\oplus (n-2)}\oplus \mathcal{O}$, such nonzero map exists only if $am-k+1\ge 0$. This completes the proof of the proposition.



\section{Proof of the estimate \eqref{eq-goal}}

The following graphs  illustrate the result in Proposition \ref{prop-3&n-1}. 
\begin{figure}[htbp]
  \centering
  \begin{subfigure}[b]{0.32\textwidth}
    \includegraphics[width=\textwidth]{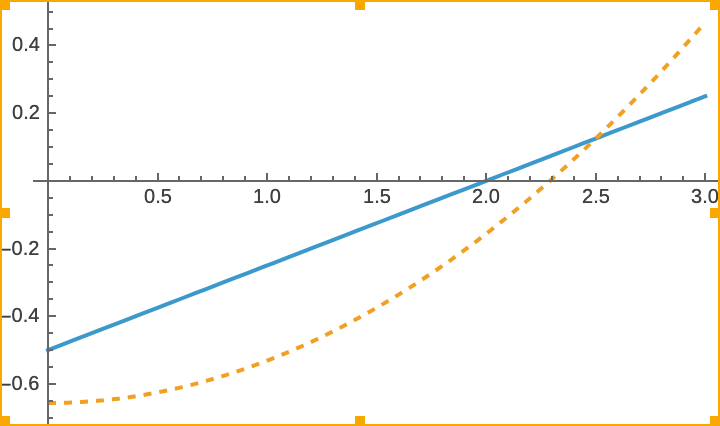}
    \caption{d=7}
  \end{subfigure}
  \hfill
  \begin{subfigure}[b]{0.32\textwidth}
    \includegraphics[width=\textwidth]{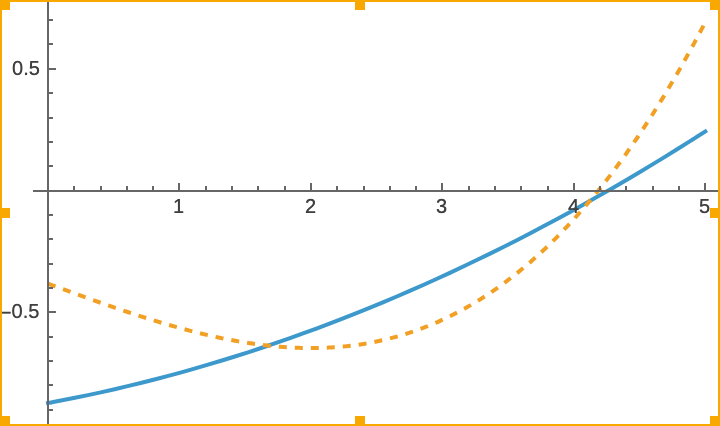}
    \caption{d=6}
  \end{subfigure}
  \hfill
  \begin{subfigure}[b]{0.32\textwidth}
    \includegraphics[width=\textwidth]{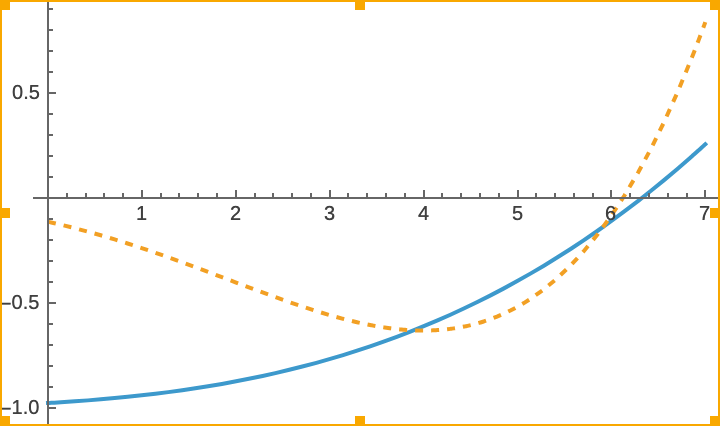}
    \caption{d=5}
  \end{subfigure}

  \vspace{1em} 

  \begin{subfigure}[b]{0.32\textwidth}
    \includegraphics[width=\textwidth]{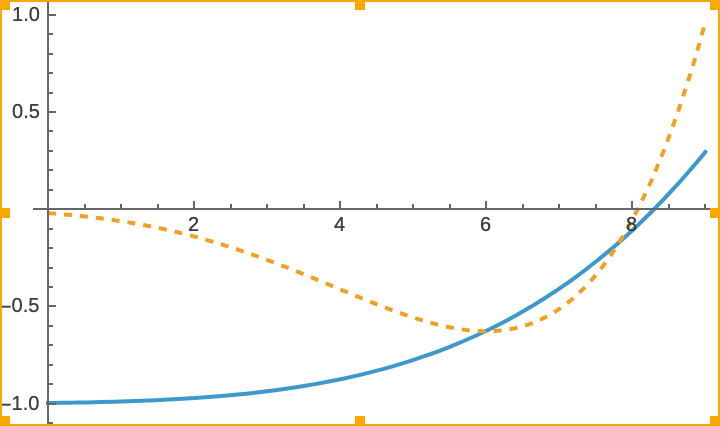}
    \caption{d=4}
  \end{subfigure}
  \hfill
  \begin{subfigure}[b]{0.32\textwidth}
    \includegraphics[width=\textwidth]{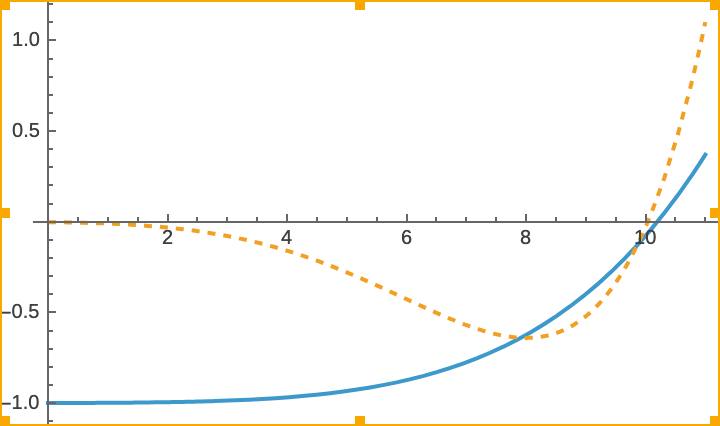}
    \caption{d=3}
  \end{subfigure}
  \hfill
  \begin{subfigure}[b]{0.32\textwidth}
    \includegraphics[width=\textwidth]{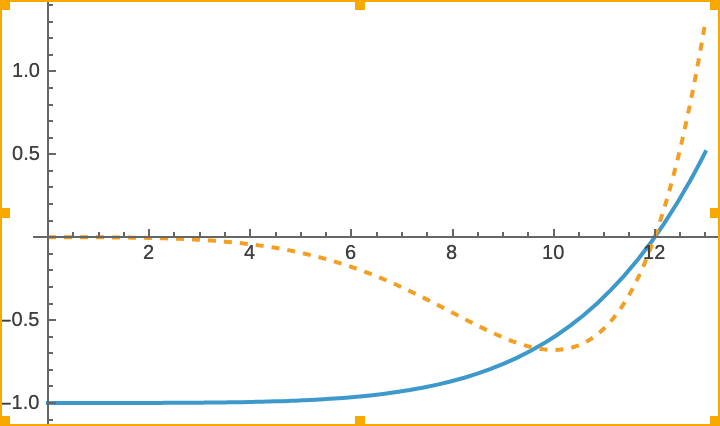}
    \caption{d=2}
  \end{subfigure}

  \caption{$n=7, 7\ge d\ge 2$}\label{fig-n7}
\end{figure}

 The solid curve (resp. the dashed curve) is the graph of $f(y):=\phi(y)/(2n^n)-1$ (resp. $g(y):=\Psi(y)/(2n^n)$) where $y=x-d$.  The first graph illustrates Example \ref{ex-d=n}. The second figure illustrates Proposition \ref{prop-n-1}. The 3rd to 5th figures illustrates that for $d\in \{5,4,3\}$ at $y_1=2(n-d+1)=x_1-d\in \{6,8,10\}$, $g'(y_1)>f'(y_1)$ and $g(y_1)>f(y_1)$ as claimed in the proof of Proposition \ref{prop-3&n-1}. The last figure illustrates Remark \ref{rmk-2nn}. The graphs are generated by the following codes.

\begin{lstlisting}
L=2;n=7;
phi[y_,n_,d_]:=L^(-n+d-1)*Sum[Binomial[n,j]*(n-d+2-j)*d^(n-j)*y^j,{j,0,n-d+1}];
Phi[y,n,d]:=L^(-n+d-1)*Sum[Binomial[n+1,j]*(n-d+3-j)/(n+1)*d^(n+1-j)*y^j,{j,0,n-d+2}];
A[n_,d_]:=d-2+L*(n-d+1);
Psi[y_,n_,d_]:=(y+d-A[n,d])*phi[y,n,d]-Phi[y_,n_,d_];
For[d=n,d>1,d--,Plot[{phi[y,n,d]/(2n^n)-1,Psi[y,n,d]/(2n^n)},{y,0,L*(n+1-d)+1}, PlotStyle->{Thick, Dashed}]//Print]

\end{lstlisting}

\begin{proof}[Proof of the estimate \eqref{eq-goal}]\label{proof-goal}
For the simplicity of notation, we set $r=n+1-d$ so that $r+d=n+1$. 
We will prove that \eqref{eq-goal} holds for $x_1=2r+d=2n-d+2$. 

\textbf{(Case I: Both $r$ and $d$ are big)}
By using the expression \eqref{eq-phix}-\eqref{eq-Phix}, we can calculate:
  \begin{eqnarray}
  && S=S(r,d)=\Phi(x_1)-\phi(x_1)\nonumber\\
  &=&\frac{2^{-(n-d+1)}n!}{(d-3)!(n-d+2)!}\int_0^d z^{d-3}(d-z)(x_1-z)^{n-d+1}(n-z)\dif z\label{eq-Psiint}\\
   &=& \frac{1}{2^r(r+1)}\sum_{j=0}^r (r+1-j)\cdot\left[ n(n-j+1)-d(d-2)\right] \frac{n!}{j!(n+1-j)!} d^{n-j}(2r)^j \label{eq-third}\\
   &=& \frac{(d+2r)^{n+1}\cdot n}{2^r(r+1)(n+1)d}\sum_{j=0}^r (r+1-j)(r-j+\Delta)\binom{n+1}{j}(1-p)^{n+1-j}p^j \label{eq-Psipro}
\end{eqnarray}
where 
\begin{equation*}
\Delta=\frac{d(r+1)}{n}, \quad 
p=\frac{2r}{d+2r}, \quad 1-p=\frac{d}{d+2r}.
\end{equation*}
For the identity (\ref{eq-third}) we use the binomial expansion of $(x_1-z)^{r}=(2r+(d-z))^r$ and integrate via the definition of Beta functions. 
To estimate the partial binomial sum, we use a method from probability theory called Cr\'{a}mer's tilt. 
Set 
$$q=\frac{r}{n+1}=\frac{r}{r+d}<p=\frac{2r}{2r+d}$$
such that the new expectation of the binomial distribution is $r=(n+1)q$. 
We can calculate: 
\begin{eqnarray*}
    \frac{(1-p)^{n+1-j}p^j}{(1-q)^{n+1-j}q^j}&=&\left(\frac{p(1-q)}{q(1-p)}\right)^{j-(n+1)q}\left(\frac{1-p}{1-q}\left(\frac{p(1-q)}{q(1-p)}\right)^q \right)^{n+1}\\
    &=&\eta^{r-j}\left(\frac{(1-p)^{1-q}p^q}{(1-q)^{1-q}q^q}\right)^{n+1}
\end{eqnarray*}
where 
\begin{eqnarray*}
    \eta=\frac{q(1-p)}{p(1-q)}=\frac{rd}{2r d}=\frac{1}{2}, \quad \frac{(1-p)^{1-q}p^q}{(1-q)^{1-q}q^q}
=2^q\frac{r+d}{2r+d}.
    \end{eqnarray*}
Substituting into the formula \eqref{eq-Psipro} and let $j=r-k$, 
we get
\begin{eqnarray}
    S &=& \frac{(n+1)^n n}{(r+1)d}\sum_{k=0}^{r}(k+1)(k+\Delta)\frac{}{}\binom{n+1}{r-k}(1-q)^{n+1-r+k} q^{r-k} 2^{-k}  \label{eq-Ssum}\\
    &=&\frac{(n+1)^n n}{(r+1)d}\sum_{k=0}^r (k+1)(k+\Delta)2^{-k} a_k 
\end{eqnarray}
where $a_k=\binom{n+1}{r-k}(1-q)^{n+1-r+k}q^{r-k}$. 
Note that for $k\geq 1$,
\begin{eqnarray*}
    \frac{a_{k}}{a_{k-1}}=\frac{r-k+1}{n-r+k+1}\cdot\frac{d}{r}\le \frac{d}{d+k}\le \frac{d}{d+1}<1.
\end{eqnarray*}
So if we set $\theta_1=\frac{d}{2(d+3)}<\frac{1}{2}$ and $\beta\coloneqq \frac{n+1}{dr}\geq \frac{n}{d(r+1)}=\Delta^{-1}$, we get the estimate:
\begin{eqnarray*}
    \frac{S}{2n^n} &\leq& \frac{(n+1)^n}{2n^n}\Delta^{-1}\sum_{k= 0}^{+\infty} (k+1)(k+\Delta)2^{-k} \cdot a_k\\
    &=&\frac{(n+1)^n}{2n^n}\Delta^{-1}a_0\left(\Delta+(1+\Delta)\frac{a_1}{a_0}+\frac{a_2}{a_0}\sum_{j=0}^{+\infty}(j+3)(j+2+\Delta)2^{-j-2}\frac{a_{2+j}}{a_2} \right)\\
    &\leq&\frac{(n+1)^n}{2n^n}\Delta^{-1}a_0 \left(\Delta+(1+\Delta)\frac{d}{d+1}+\frac{d^2}{4(d+1)(d+2)}\sum_{j=0}^{+\infty}(j+3)(j+2+\Delta)2^{-j}\left(\frac{d}{d+3}\right)^j \right)\\
    &=& \frac{(n+1)^n}{2n^n} \Delta^{-1}a_0 \left(\alpha_1\Delta+\tau_1\right)
\end{eqnarray*}
where $\alpha_1=1+\frac{d}{d+1}+\frac{d^2}{4(d+1)(d+2)}\frac{3-2\theta_1}{(1-\theta_1)^2}$ and $\tau_1=\frac{d}{d+1}+\frac{d^2}{2(d+1)(d+2)}\frac{\theta_1^2-3\theta_1+3}{(1-\theta_1)^3}$.
By using Robbins' version of the Stirling approximation:
\begin{eqnarray*}
    e^{\frac{1}{12n+1}}<\frac{n!}{n^n e^{-n}\sqrt{2\pi n}}< e^{\frac{1}{12 n}},
\end{eqnarray*}
one can verify that: with $r=(n+1)q$,  
\begin{eqnarray*}
    a_0&=&\binom{n+1}{r}(1-q)^{n+1-r}q^r\le 
    \frac{1}{\sqrt{2\pi (n+1)q (1-q)}}\\   
    &=&
    \sqrt{\frac{(n+1)}{(2\pi) rd}}=\sqrt{\frac{\beta}{2\pi}}. 
\end{eqnarray*}
Thus, when $d\leq r$, we have $\beta=d^{-1}+r^{-1}\leq 2d^{-1}$, then
\begin{eqnarray*}
     \frac{S}{2n^n}
     \leq  \frac{e}{2\sqrt{2\pi}}\sqrt{\beta}
     \left(\alpha_1+\tau_1\beta\right)\leq \frac{e}{2\sqrt{2\pi}}\sqrt{(2/d)}\left(\alpha_1+\tau_1\cdot (2/d)\right)\leq 1
\end{eqnarray*}
provided if $d\geq 9$. 

When $r\leq d$, similarly
\begin{eqnarray*}
\frac{S}{2n^n}&\leq& \frac{(n+1)^n}{2n^n}\Delta^{-1}a_0\left(\Delta+(1+\Delta)\frac{a_1}{a_0}+\frac{a_2}{a_0}\sum_{j=0}^{+\infty}(j+3)(j+2+\Delta)2^{-j-2}\frac{a_{2+j}}{a_2} \right)\\
   &\leq& \frac{(n+1)^n}{2n^n}\Delta^{-1}a_0\left(\Delta+(1+\Delta) +\frac{r-1}{r}\sum_{j=0}^{+\infty}(j+3)(j+2+\Delta)2^{-j-2}\left(\frac{r-2}{r} \right)^j\right) \\
     &=&\frac{(n+1)^n}{2n^n}\Delta^{-1}a_0\left(\alpha_2 \Delta+\tau_2\right)
\end{eqnarray*}
where $\theta_2=\theta_2(r)=\frac{r-2}{2r}$ and $\alpha_2=2+\frac{r-1}{4r}\frac{3-2\theta_2}{(1-\theta_2)^2},\, \tau_2=1+\frac{r-1}{2r}\frac{\theta_2^2-3\theta_2+3}{(1-\theta_2)^3}$. 
Use $\Delta^{-1}\le \frac{r}{r+1}\beta$ with $\beta=d^{-1}+r^{-1}\le 2r^{-1}$. We see the inequality
\begin{eqnarray*}
    \frac{S}{2n^n}\leq\frac{e}{2\sqrt{2\pi}}\left(\alpha_2 +\frac{r}{r+1}\tau_2 \cdot (2/r) \right)\sqrt{2/r}\leq1
\end{eqnarray*}
 holds provided if $r\ge 11$.

\textbf{(Case II: $r$ small and $d$ large)}
    Recall the relation $n+1=r+d$. 
To illustrate our way to estimate, consider the case $r=n-d+1=3$ and $d\ge 3$ so that $n\ge 5$. 
Then we can calculate and estimate: 
\begin{eqnarray*}
    \frac{\Psi(x_1)-\phi(x_1)}{2n^n}&=& \frac{(n-2)^{n-2}}{16n^n(1+n)}(115n^3-126n^2-66n+40)\\
    &\le& \frac{e^{-2}}{16}\frac{115 n^3-126 n^2-66n+40}{(n-2)^2(1+n)}\\
    &=& \frac{115 e^{-2}}{16}\left(1+\frac{1}{115}\frac{{219} n^2-{66}n-{420}}{(n-2)^2(n+1)}\right)\\
    &=& \frac{115 e^{-2}}{16} \left(1+\frac{1}{115}\frac{m}{m+3}(219 m^{-1}+810 m^{-2}+324 m^{-3})\right)
\end{eqnarray*}
where for the first inequality we used $(n-2)^n/n^n\le e^{-2}$ and for the last equality we set  $m=n-2$. 
One can verify numerically that the last expression is less than 1 when $m=n-2\ge 69$. So we can just verify numerically for $6\le n\le 70$ or equivalently $3\le d\le 68$.

For $r\ge 4$, we use the formula \eqref{eq-Ssum} for $S=S(r,d)$ to estimate (note that $\Delta=\frac{d(r+1)}{n}\le r+1$):  
\begin{eqnarray*}
    &&\frac{S}{2n^n}
    =\frac{(n+1)^n n}{2 n^n (r+1)d}\sum_{k=0}^{r}(k+1)(k+\Delta)\frac{}{}\binom{n+1}{r-k}(1-q)^{n+1-r+k} q^{r-k} 2^{-k}    
    \\
    &\le &\frac{(n+1)^n}{2n^n}\frac{1}{r+1}\left(1+\frac{r-1}{d}\right)\sum_{k=0}^r \frac{(k+1)(k+r+1)(n+1)!}{(r-k)!(n+1-(r-k))!}(\frac{d}{n+1})^{d+k}(\frac{r}{n+1})^{r-k}2^{-k}\\
    &\le& \frac{e}{2}\frac{1}{r+1}(1+\epsilon)\sum_{k=0}^r(k+1)(k+r+1)\frac{1}{(r-k)!}r^{r-k}2^{-k}(1+\frac{r}{d})^{-d}
\end{eqnarray*}
where we set $\epsilon=r/d$. Moreover, we have $1+\epsilon< e^\epsilon$ and
\begin{eqnarray*}
    -d\cdot \log (1+\epsilon)=-d (\epsilon-\frac{\epsilon^2}{2}+\cdots)\le -d\epsilon+d\frac{\epsilon^2}{2}=-r+\frac{r\epsilon}{2}.
\end{eqnarray*}
Then    \begin{eqnarray*}
     \frac{S}{2n^n} &\le& \frac{e}{2(r+1)}(1+\epsilon)\sum_{k=0}^r \frac{(k+1)(k+r+1)r^{r-k}}{2^k(r-k)!}e^{-r+r\epsilon/2}.
    \end{eqnarray*}
We define
\begin{equation}
    R_\infty=R_\infty(r)=\frac{e^{1-r}}{2(r+1)}\sum_{k=0}^r \frac{(k+1)(k+r+1)r^{r-k}}{2^k(r-k)!}=\lim_{d\rightarrow+\infty} \frac{S}{2n^n}.
\end{equation} 
We have $S/2n^n< R_{\infty} e^{r(r+2)/(2d)}$.
So $S< 2n^n$ when $d\ge d(r)=\frac{r(r+2)}{2(-\log R_\infty)}$.
By \textbf{(Case I)}, we can assume that $4\le r\le 10$. 
By the numerical analysis (we provide the following Wolfram Mathematica code), we can determine the choice of $d(r)$ for each $4\leq r \leq 10$. 
Moreover, by the beginning discussion of \textbf{(Case II)}, we can assume $d(3)=68$. We list the results in Table \ref{tab:d(r)}. 

\begin{lstlisting}
R = (E^(1 - r)/(2*(r + 1)))*Sum[((k + 1)*(k + r + 1)*r^(r - k))/
         (2^k*(r - k)!), {k, 0, r}]
Table[N[(r*(r + 2))/(2*(-Log[R]))], {r, 4, 10}]
\end{lstlisting}

\begin{table}[!htp]
    \centering
    \caption{$d(r)$}
    \label{tab:d(r)}
   \begin{tabular} {|c|c|c|c|c|c|c|c|c|c|c|c|c|c|} 
\hline
   $r$ & 3
   & 4 & 5& 6& 7 & 8& 9& 10 
   \\
   \hline
   $d(r)$ & 68
   & 101 & 89 & 91 & 97 
   & 106 & 117 & 128 
   \\ 
   \hline
\end{tabular}
\end{table}

(\textbf{Case III: $d$ small and $r$ large})
We use the integral formula \eqref{eq-Psiint} to estimate.
We use 
\begin{eqnarray*}
2^{-(n-d+1)}(x_1-z)^{n-d+1}=2^{-r}(2r+d-z)^r=r^r (1+\frac{d-z}{2r})^r\le r^r e^{(d-z)/2}
\end{eqnarray*}
and $n-z\le n$ to get:
\begin{eqnarray*}
    \frac{S}{2n^n}&=&\frac{2^{-(n-d+1)}n!}{2n^n (d-3)!(n-d+2)!}\int_0^d z^{d-3}(d-z)(x_1-z)^{n-d+1}(n-z)\dif z\\
    &\le& \frac{n(n-1)\cdots (n-d+3)}{(d-3)!}\cdot \frac{r^r\cdot n}{2n^n}\int_0^d z^{d-3}(d-z) e^{(d-z)/2}\dif z\\
    &=& \frac{n(n-1)\cdots (n-d+3)}{(d-3)!n^{d-2}}\frac{(n-d+1)^r}{2 n^r}\int_0^d z^{d-3}(d-z) e^{(d-z)/2}\dif z\\
    &\le& (1+\frac{d-1}{r})^{-r}\frac{1}{2(d-3)!}\int_0^d z^{d-3}(d-z) e^{(d-z)/2}\dif z
    \end{eqnarray*}
    where $\epsilon=\frac{d-1}{r}$ and 
    \begin{eqnarray*}
    -r \log(1+\epsilon)=-r \sum_{k=0}^{+\infty} (\epsilon-\frac{\epsilon^2}{2}+\frac{\epsilon^3}{3}+\cdots)\le -r\epsilon+r \epsilon^2/2.
\end{eqnarray*}
Then
 \begin{eqnarray*}
    \frac{S}{2n^n}&\le& e^{r \epsilon^2/2}e^{1-d/2}\frac{1}{2(d-3)!}\int_0^d z^{d-3}(d-z)e^{-z/2}\dif z
\end{eqnarray*}
We define
\begin{eqnarray*}
    R_\infty(d)&=&e^{1-\frac{d}{2}}\frac{1}{2(d-3)!}\int_0^d z^{d-3}(d-z)e^{-z/2}\dif z=\mathop{\lim}_{r\rightarrow+\infty} \frac{S}{2n^n}. 
\end{eqnarray*}
Therefore, we have $S/2n^n< e^{\frac{(d-1)^2}{2r}}R_\infty(d)$. So for $S< 2n^n$, it suffices $r\ge r(d):=\frac{(d-1)^2}{2(-\log R_\infty(d))}$. By \textbf{Case I}, we can assume $3\le d\le 8$ and the corresponding $r(d)$ can be solved numerically (see the following Wolfram Mathematica code): 
\begin{lstlisting}
R = E^(1 - d/2)*(1/(2*(d - 3)!))*
     Integrate[(z^(d - 3)*(d - z))/E^(z/2), {z, 0, d}]
Table[N[(d - 1)^2/(2*(-Log[R]))], {d, 3, 8}]
\end{lstlisting}
\begin{table}[!htp]
\caption{$r(d)$}
    \label{tab:r(d)}
    \centering
\begin{tabular}{|c|c|c|c|c|c|c|c|c|c|}
\hline
   $d$ & 3 & 4 & 5& 6& 7 & 8\\
   \hline
   $r(d)$ & 16 & 20 & 27 & 34 & 43 & 52\\
   \hline
\end{tabular}    
\end{table}

(\textbf{Case IV: Both $d$ and $r=n+1-d$ are small}).
We can numerically check (see the following Wolfram Mathematica code) that the estimate (\ref{eq-goal}) holds for the following finitely many cases:
\begin{itemize}
    \item $3\le d\le 8$ and $3\le r\le r(d)-1$. 
\item
$4\le r\le 10$ and $3\le d\le d(r)-1$. 
\end{itemize}
Both $r(d)$ and $d(r)$ are given by Table \ref{tab:d(r)} and \ref{tab:r(d)}.
\begin{lstlisting}
x1 = 2 r + d
n = r + d - 1
S = (n!/(2^r*(d - 3)!*(n - d + 2)!))*
  Integrate[z^(d - 3)*(d - z)*(x1 - z)^r*(n - z), {z, 0, d}]
Do[Print[S/(2*n^n) < 1], {d, 3, 8}, {r, 3, 52}]
Do[Print[S/(2*n^n) < 1], {r, 3, 10}, {d, 3, 128}]
\end{lstlisting}
\end{proof}

\section{A proof of the formula \eqref{eq-Qvol}}

We calculate the Duistermaat-Heckman measure via the localization method. 
Let $\omega=\omega_{\mathrm{FS}}$ denote the restriction of the Fubini-Study metric of $\IP^{n+1}$ on $X=Q^n=\{Z_0Z_1+Z_2Z_3+Z_4^2+\cdots+Z_{n+1}^2=0\}\subset \IP^{n+1}$ for $n\ge 3$. The $\IC^*$ action on $X$ is given by: for any $\lambda\in \IC^*$, 
\begin{equation*}
    \lambda\circ (Z_0, Z_1, Z_2, Z_3, Z_4, \dots, Z_{n+1})=(Z_0, \lambda^2 Z_1, Z_2, \lambda^2 Z_3, \lambda Z_4, \dots, \lambda Z_{n+1}).
\end{equation*}
The corresponding $S^1$ action is a Hamiltonian with moment map given by:
\begin{equation*}
    \mu([Z])=\frac{2|Z_1|^2+2|Z_3|^2+|Z_4|^2+\cdots+|Z_{n+1}|^2}{|Z_0|^2+|Z_1|^2+\cdots +|Z_{n+1}|^2}.
\end{equation*}
We introduce the equivariant volume functional:
\begin{equation*}
    V(\theta)=\int_Xe^{i\mu(z) \xi}\frac{\omega^n}{n!}=\int_\IR e^{i\theta \xi}\, \mathrm{DH}(\xi)
\end{equation*}
where $\mathrm{DH}(\xi)=\mu_* \frac{\omega^n}{n!}$ is the Duistermaat-Heckman measure. As a consequence, the density function of $\mathrm{DH}(\xi)$ is the Fourier transform of $V(\theta)$. 
    
The localization formula calculates the integral $V(\theta)$ via data on the fixed points sets (including the values of moment map and equivariant Euler curvature forms of normal bundles) which consists of three components in this case:
\begin{eqnarray*}
    && F_0=\mu^{-1}(0)\cong \IP^1: [*, 0, *, 0, \dots, 0],\\
    && F_1=\mu^{-1}(2)\cong \IP^1: [0, *, 0, *, 0, \dots, 0],\\
    && F_2=\mu^{-1}(1)\cong Q^{n-4}: [0, 0, 0, 0, Z_4, \dots, Z_{n+1}] \text{ with } Z_4^2+\cdots Z_{n+1}^2=0.  
\end{eqnarray*}
The contribution from $F_0$ is given by:
\begin{eqnarray*}   V_0(\theta):=\int_{\IP^1}\frac{e^{\omega}}{((\omega-i\theta))^{n-2}(-i 2\theta)}=\frac{(-1)^{n+1}2\pi}{2}((i\theta)^{1-n}+(n-2)(i\theta)^{-n}). 
\end{eqnarray*}
We have the following useful formula: for any $a\in \IR$, 
\begin{eqnarray*}
    \frac{1}{2\pi}\int_\IR\frac{e^{i a \theta\xi}}{(i\theta)^k}=\frac{(a-\xi)^{k-1}}{2(k-1)!}
\mathrm{Sign}(a-\xi). \end{eqnarray*}
So we can calculate the Fourier transform of $V_0(\theta)$:
\begin{eqnarray*}
    \frac{1}{2\pi}\int_{\IR} V_0(\theta) e^{-i\theta \xi}\dif\theta&=&\frac{1}{4(n-1)!}((n-2)\xi^{n-1}-(n-1)\xi^{n-2})\mathrm{Sign}(-\xi)\\
    &=:& \frac{f(\xi)}{2} \mathrm{Sign}(-\xi).
\end{eqnarray*}
The contribution from $F_2$ is given by:
\begin{eqnarray*}
V_2(\theta):=\int_{\IP^1}\frac{e^{(\omega+2i\theta)}}{((\omega+i\theta))^{n-2}(i2\theta)}=\frac{2\pi }{2}e^{2i\theta}((i\theta)^{1-n}+(2-n)(i\theta)^{-n}).
\end{eqnarray*}
Similar to above, the Fourier transform of $V_2(\theta)$ is equal to:
\begin{eqnarray*}
    \frac{1}{2\pi}\int_{\IR} V_2(\theta) e^{-i\theta \xi}\dif\theta&=&\frac{1}{4(n-1)!}(-(n-2)(2-\xi)^{n-1}+(n-1)(2-\xi)^{n-2})\mathrm{Sign}(2-\xi)\\
    &=:& \frac{g(\xi)}{2} \mathrm{Sign}(2-\xi). 
\end{eqnarray*}
The contribution from $F_1$ is given by:
\begin{eqnarray*}    V_1(\theta)&:=&\int_{Q^{n-4}}\frac{e^{(\omega+i\theta)}}{(\omega+i\theta)^2(\omega-i\theta)^2}=\int_{Q^{n-4}}\frac{e^\omega e^{i\theta}}{(\omega^2+\theta^2)^2}\\
&=&\sum_{k=0}^{\lfloor\frac{n-4}{2}\rfloor}\frac{(k+1)}{(n-4-2k)!}\frac{e^{i\theta}}{(i\theta)^{2k+4}}.
\end{eqnarray*}
The Fourier transform of $V_1(\theta)$ is equal to:
\begin{eqnarray*}
    \frac{1}{2\pi}\int_{\IR} V_1(\theta) e^{-i\theta \xi}\dif\theta&=&\sum_{k=0}^{\lfloor \frac{n-4}{2}\rfloor}\frac{(k+1)}{2(n-4-2k)!(2k+3)!}(1-\xi)^{2k+3} \mathrm{Sign}(1-\xi). 
\end{eqnarray*}
We claim that this is equal to $-\frac{f(\xi)+g(\xi)}{2}\mathrm{Sign}(1-\xi)$ so that the Fourier transform is $V(\theta)$ is equal to:
\begin{eqnarray*}
    \rho(\xi)&=&\frac{f(\xi)}{2}\mathrm{Sign}(-\xi)+\frac{g(\xi)}{2}\mathrm{Sign}(2-\xi)-\frac{f(\xi)+g(\xi)}{2}\mathrm{Sign}(1-\xi)\\
    &=& 
    \left\{\begin{array}{ll}
    -f(\xi) & 0\le \xi\le 1\\
    g(\xi) & 1\le \xi\le 2\\
    0 & \mathrm{otherwise}.
    \end{array}\right.
\end{eqnarray*}
It is immediate to check that $\rho(\xi)=-\frac{1}{n!}\frac{d}{d\xi} \vol(H-\xi E)$ by using the expression from \eqref{eq-Qvol}, which verifies the formula \eqref{eq-Qvol} after integration. 
To verify the claim, we calculate:
\begin{eqnarray*}
    \xi^{n-1}-(2-\xi)^{n-1}&=&(1-(1-\xi))^{n-1}-(1+(1-\xi))^{n-1}=\sum_{k=0}^{n-1}\binom{n-1}{k}((-1)^{k}-1)(1-\xi)^k\\
    &=&-\sum_{j=0}^{\lfloor \frac{n-2}{2}\rfloor} \binom{n-1}{2j+1} (1-\xi)^{2j+1},
\end{eqnarray*}
and finally verify the claimed identity:
\begin{eqnarray*}
    -2 (n-1)!(f(\xi)+g(\xi))&=& (n-2)(\xi^{n-1}-(2-\xi)^{n-1})-(n-1)(\xi^{n-2}-(2-\xi)^{n-2}),\\
    &=& \sum_{j=0}^{\lfloor \frac{n-2}{2}\rfloor} (n-2)\binom{n-1}{2j+1}(1-\xi)^{2j+1}-\sum_{j=0}^{\lfloor \frac{n-3}{2}\rfloor}(n-1)\binom{n-2}{2j+1}(1-\xi)^{2j+1}\\
    &=&\sum_{j=0}^{\lfloor \frac{n-2}{2}\rfloor}\frac{(n-1)!2j}{(2j+1)!(n-2-2j)!}(1-\xi)^{2j+1}\\
    &=&2(n-1)! \sum_{k=0}^{\lfloor \frac{n-4}{2}\rfloor}\frac{k+1}{(n-4-2k)!(2k+3)!}(1-\xi)^{2k+3}. 
    \end{eqnarray*}


\section*{Postscript Notes} Since the appearance of the first version of this work in June 2025, there have been related developments. In \cite{LMZ26}, the authors, jointly with Kewei Zhang, used a similar strategy of minimal rational curves to prove sharp volume gap estimates for all K\"{a}hler manifolds with positive lower bound on Ricci curvature. In October 2026, OpenAI \cite{OAI26} claims to solve the ODP conjecture \ref{conj:ODP} in all dimensions. However, the sharp global volume bounds for K-semistable smooth Fano manifolds proved in this paper is NOT covered by its results.

\bibliographystyle{alpha}
\bibliography{ref}

\end{document}